\documentclass[reqno]{amsart}
\usepackage{hyperref}
\usepackage[a4paper, left=3cm, right=3cm, top=4cm, bottom=4cm]{geometry}
\vspace{9mm}
\begin{document}
\title[\hfilneg \hfil Stochastic controllability of fractional system]{Stochastic controllability for a non-autonomous fractional neutral differential equation with infinite delay in abstract space}
\author[A. Khatoon,  A. Raheem  \& A. Afreen \hfil \hfilneg]
{A. Khatoon$^{*}$,   A. Raheem \& A. Afreen}

\address{Areefa Khatoon \newline
	Department of Mathematics,
	Aligarh Muslim University,\newline Aligarh, U. P. -
	202002, India} \email{areefakhatoon@gmail.com}

\address{Abdur Raheem \newline Department of Mathematics,
	Aligarh Muslim University,\newline Aligarh, U. P. -
	202002, India} \email{araheem.iitk3239@gmail.com}

\address{Asma Afreen \newline
	Department  of Mathematics,
	Aligarh Muslim University,\newline Aligarh, U. P. -
	202002, India} \email{afreen.asma52@gmail.com}

\renewcommand{\thefootnote}{} \footnote{$^*$ Corresponding author:
\url{ A. Khatoon (areefakhatoon@gmail.com)}}

\subjclass[2010]{26A33, 93B05, 34K05, 34K35, 47H08.}

\keywords{Non-autonomous, fractional derivative,  controllability, evolution family, neutral, measure of non-compactness.}

\begin{abstract}
This paper deals with the controllability for a class of non-autonomous neutral differential equations of fractional order with infinite delay in an abstract space. The semi-group theory of bounded linear operators, fractional calculus, and stochastic analysis techniques have been implemented to achieve the main result. We prove the existence of mild solution and controllability of the system by using the theory of measure of non-compactness, fixed point theorems, and $k$-set contractive mapping. An example is given to demonstrate the effectiveness of the abstract result.
\end{abstract}

\maketitle \numberwithin{equation}{section}
\newtheorem{theorem}{Theorem}[section]
\newtheorem{lemma}[theorem]{Lemma}
\newtheorem{proposition}[theorem]{Proposition}
\newtheorem{corollary}[theorem]{Corollary}
\newtheorem{remark}[theorem]{Remark}
\newtheorem{definition}[theorem]{Definition}
\newtheorem{example}[theorem]{Example}
\allowdisplaybreaks

\section{\textbf{Introduction}}
The idea of fractional derivatives was first introduced in Leibniz's letter to L'Hospital on 30th September 1695 \cite{g1}, when he raised the meaning of derivative of order $\frac{1}{2}$. The issue raised by Leibniz attracted many well-known mathematicians, including Liouville, Grünwald, Riemann, Euler, Lagrange, Heaviside, Fourier, Abel, Letnikov, and many others. Since the 19th century, the theory of fractional calculus originated rapidly and was the beginning of many disciplines, including fractional differential equations, non-integer order geometry, and fractional dynamics. Nowadays, there are numerous applications in various branches, such as optimal control, porous media, fractional filters, signal and image processing, soft matter mechanics, fractals etc.The main reason for the application of fractional calculus becoming more popular is that the non-integer order model describes a more accurate model than the integer-order. For the basics of fractional calculus and its applications, we refer to the papers \cite{r5, y1, y2, EG2001} and references therein.
   
	A stochastic differential equation is one in which one or more of the terms are random variables. Stochastic differential equations are used to describe various phenomena, notably unstable stock prices and physical devices with thermal fluctuations, population dynamics, biology, weather prediction model, molecular dynamics, and the textile industry. A brief summary of stochastic theory can be seen in  \cite{8,14}.

In 1960, Kalman \cite{K1963}, proposed the idea of controllability. Controllability theory aims the ability to control a specific system to the desired state by providing appropriate input functions during a finite time interval. Controllability theory is a vital concept in the dynamical control system and plays a significant role in examining various dynamical control processes. Several papers on the controllability of systems represented by differential equations in abstract spaces have been published.

In the setting of integer order systems, Balachandran et al. \cite{BD2002}, gave a survey paper for controllability of nonlinear control systems in Banach spaces in 2001. In \cite{ M2001, MZ2003}, Mammudov et al. used fixed point theorems to obtain the necessary requirements for the controllability of linear and nonlinear stochastic systems. Using sequential approach, Shukla et al. \cite{SSP2015}, established some sufficient conditions for approximate controllability of semilinear systems in 2015. In 2020, Kumar et al. \cite{A2020}, investigated the approximate controllability of second-order nonlinear differential equations with finite delay using Schauder’s fixed-point theorem.

In the setting of a fractional-order system,  Tai and Wang \cite{TW2009}, studied the controllability result for a fractional impulsive neutral differential equation in a Banach space using the Krasnoselskii's fixed point theorem in 2009. In 2020, Chen et al.  \cite{C2020}, used Schauder's fixed point theorem to investigate the existence of mild solutions and approximate controllability for the fractional evolution equation. In 2022, Raheem and Kumar \cite{AM2022}  proved the existence and uniqueness of the mild solution of the fractional system with a deviated argument. They used the measure of non-compactness and the M\"{o}nch fixed point theorem to investigate the controllability result. 

We observe that all previous works have mainly focused on the case when differential operators in the main parts are independent of time, implying that the differential system is autonomous  \cite{D2014, AS2016,  K2018, W2009, L2013}. However, when treating some parabolic equations, where the differential operators depend on time or a non-autonomous system appears frequently in the application, see the papers \cite{ A2020, AM2022, X2017, PYX2021, LTX2017}. There are only a few papers dealing with the controllability of fractional non-autonomous system.

More precisely, in 2017, Chen et al. \cite{PXY2017} investigated the existence of mild solutions for initial value problem to the following nonlinear time
fractional non-autonomous evolution equations with delay in Banach space
\begin{eqnarray*}
	\left \{ \begin{array}{lll} ^C{D}_t^{\alpha}u(t)+{A}(t)u(t)=f\left(t,u(\tau_1(t)),\ldots,u(\tau_m(t))\right),\quad ~t \in [0,T],
		\vspace{0.1cm}
		&\\u(0)=u_0,
	\end{array}\right.\hspace{2.4cm}
\end{eqnarray*}
where $^C{D}_t^{\alpha}$ is the standard Caputo’s fractional time derivative of order $0<\alpha\leq 1.$ 

In 2018, Chen et al. \cite{C2018} considered the following nonlinear time fractional non-autonomous evolution equation with nonlocal conditions 
\begin{eqnarray*}
	\left \{ \begin{array}{lll} ^C{D}_t^{\alpha}u(t)+{A}(t)u(t)=f\left(t,u(t)\right),\quad ~t \in [0,T],
		\vspace{0.1cm}
		&\\u(0)=A^{-1}(0)g(u),
	\end{array}\right.\hspace{2.4cm}
\end{eqnarray*}
and investigated the existence of mild solutions by using measure of noncompactness and fixed point theorem.

In 2020, Kumar et al. \cite{AHCKD2020} considered the non-autonomous fractional differential equations with integral impulse condition and proved the existence of mild solutions by using measure of non-compactness, fixed point theorems, and k-set contraction. 

Motivated by the works of \cite{AHCKD2020, AM2022} and \cite{SMN2012}, we study the controllability to non-autonomous stochastic system of order $0<\gamma\leq1.$ We consider the following non-autonomous fractional stochastic neutral differential equation of order $0<\gamma\leq1$ in a real separable Hilbert space ${X}$:

\begin{eqnarray}\label{1.1}
	\left \{ \begin{array}{lll} ^C{D}^{\gamma}\Big[\vartheta(t)-f\big(t,\vartheta_t\big)\Big]+{A}(t)\Big[\vartheta(t)-f\big(t,\vartheta_t\big)\Big]&\\\hspace{2cm}={B}\upsilon(t)+ g\big(t,\vartheta_t,\vartheta(t)\big)+h\big(t,\vartheta(t)\big)\dfrac{dw(t)}{dt}, \quad t \in J=[0,\ell],
		&\\  \vartheta_0(t)={A}^{-1}(0)\phi(t) \in L_2(\Omega,\wp),~ t\in J_0= (-\infty,0],
	\end{array}\right.
\end{eqnarray}
where $^C{D}^{\gamma}$ denotes the Caputo fractional derivative of order $0<\gamma \leq 1$. Let $Y$ be another separable Hilbert spaces and the notation $\|\cdot\|$ denotes the norm of the spaces $X, Y, \mbox{ and } L(Y,X)$. ${A}(t)$ is a densely defined closed linear operator on $X$ and $\vartheta: J\rightarrow X$ is the state function. The function $\vartheta_t:(-\infty,0]\rightarrow  \mathrm{X},~ \vartheta_t(s)=\vartheta(t+s)$ belongs to the abstract phase space $\wp$. The control function $\upsilon$ takes value in a separable Hilbert space $U$  and the linear operator $B \in L(U,X)$ is bounded. The functions $f, g$ and $h$ satisfy some conditions. 

This paper is planned as follows. Section 1 contains the introduction and considered problem. Section 2 includes several important definitions, assumptions, and lemmas. The main result for the controllability of a non-autonomous fractional system is established in Section 3. In the last section, an example is provided as an application.

\section{\textbf{Preliminaries and Assumptions}}
Let $(\Omega, \Upsilon, P)$ be a complete probability space.  $\Upsilon_t,~t\in J$ is a normal filtration, which is a right continuous increasing family and $\Upsilon_0$ contains all $P$-null sets. Assume that $\{e_n\}_{n=1}^{\infty}$ is a complete orthonormal basis in $Y$. $\{w(t)\}_{t\geq0}$ denotes $Q$-Wiener process on $(\Omega, \Upsilon, P)$  with a bounded nuclear covariance operator $Q$ such that $Tr(Q)=\sum_{i=1}^{\infty}\lambda_i$ such that $Qe_n=\lambda_ne_n, ~n=1,2,\dots$, where $\lambda_n$ denotes the bounded sequence of non-negative real numbers. Thus, we get $w(t)=\sum_{n=1}^{\infty}\sqrt{\lambda_n}w_n(t)e_n$, where $w_n(t),~n=1,2,\dots$ are  mutually independent one-dimensional standard Wiener motions over the probability space $(\Omega, \Upsilon, P)$. For all $\psi\in L(Y,X),$ we define
\begin{eqnarray*}
	\|\psi\|^2_Q=Tr(\psi Q \psi^*)=\sum\limits_{n=1}^{\infty}\|\sqrt{\lambda_n}\psi e_n\|^2.
\end{eqnarray*} 
If the value of $\|\psi\|^2$ is finite, then $\psi$ is said to be $Q$-Hilbert Schmidt operator. Let the space of all $Q$-Hilbert-Schmidt operators $\psi:Y\rightarrow X$ be denoted as $L^0_2(Y,X)$. The completion 
$L^0_2(Y,X)$ of $L(Y,X)$ with respect to the norm $\|\cdot\|_Q$, where $\|\psi\|^2_Q=<\psi,\psi>$ forms a Hilbert space  with the above norm.

We will use an axiomatic formulation of the phase space $\wp$ described by Hale and Kato \cite{JK1978}. The axioms of the phase space $\wp$ are demonstrated for $\Upsilon_0$-measurable functions from $J$ into $X$ endowed with the seminorm defined as $\|\cdot\|_{\wp}$. We assume that $\wp$ satisfies the following axioms:
\begin{itemize}
	\item[(A1)] If $\vartheta:(-\infty, c) \rightarrow {X},~c>0$ is continuous on $[0,c)$ and $\vartheta_0 \in \wp,$ then for every $t \in [0,c),$
	the following conditions hold:
	\begin{itemize}
		\item [(i)] $\vartheta_t$ is in $\wp.$
		
		\item[(ii)] $\|\vartheta(t)\| \leq K_1 \|\vartheta_t\|_{\wp}.$
		
		\item[(iii)] $\|\vartheta_t\|_{\wp} \leq K_2(t)\sup\big\{\|\vartheta(s)\|:0 \leq s \leq t\big\}+K_3(t)\|\vartheta_0\|_{\wp},$
		where $K_1>0$ is a constant, $K_2, K_3:[0,\infty) \rightarrow [0, \infty), $ ~$K_2(\cdot)$ is continuous and $K_3(\cdot)$ is locally bounded.
	\end{itemize}
	\item[(A2)] For $\vartheta(\cdot)$ in (A1), $\vartheta_t$ is  $\wp$-valued functions in $[0,c).$ 
	\item[(A3)] The phase space $\wp$ is complete.
\end{itemize}
The stochastic process $\vartheta_t : \Omega \rightarrow \wp$, where $t\in [0,\infty)$ is defined as $\vartheta_t=\big\{\vartheta(t+s)(\omega): s \in (-\infty, 0]\big\}.$

Let $L^2(\Omega,\mathrm{X})$ be defined as the collection of all strongly measurable, square integrable, $X$-valued  random variables, with the norm $\big\|\vartheta(\cdot)\big\|_{L^2}=\left(E \big\|(\cdot,\omega)\big\|^2\right)^{1/2},$ where the expectation $E$ is defined as $E(h)=\displaystyle\int_{\Omega}h(\omega)dP.$ Let $J_0=(-\infty,\ell)$ and $C(J_0,L^2(\Omega,{X}))$ denotes the Banach space of all $\Upsilon_t$ adapted measurable, $X$-valued stochastic process that are continuous from $J$ to $L^2(\Omega,{X})$ with the condition $\sup\limits_{t \in J_0}E\big\|\vartheta(t)\big\|^2 <\infty.$

Let $Z$ be the closed subspace of all continuous stochastic processes $\vartheta$ that belongs to space $C(J_0,L_2^{\Upsilon}(\Omega,\wp))$ consisting of measurable and $\Upsilon_t$ adapted processes such that ${A}^{-1}(0)\phi \in \wp$. The space $Z$ forms a Banach space with the norm defined by
\begin{eqnarray*}
	\|\vartheta \|_Z= \left(\sup\limits_{t \in J}E \big\|\vartheta_t\big\|_{\wp}^2\right)^{\frac{1}{2}},
\end{eqnarray*} where
$\|\vartheta_t\|_{\wp} \leq \tilde{K_3}E\|{A}^{-1}(0)\phi\|_{\wp}+\tilde{K_2}\sup\big\{E\|\vartheta(s)\|:0 \leq s \leq b\big\},~\tilde{K_3}=\sup\limits_{t \in J}K_3(t),~\tilde{K_2}=\sup\limits_{t\in J}K_2(t).$

Let $L^2(\Upsilon,X)$ denotes the Banach space of all square integrable, $\Upsilon$-measurable random variables. The control function $\upsilon(\cdot)\in L^2_{\Upsilon}([0,\ell],U),$ where $L^2_{\Upsilon}([0,\ell],U)$ is a Banach space with the norm
\begin{eqnarray*}
	\|\upsilon\|_{L^2_{\Upsilon}}=\left(E\int_{0}^{\ell} \big\|\upsilon(t)\big\|^2_Udt\right)^{\frac{1}{2}},
\end{eqnarray*}
satisfying $E\displaystyle\int_{0}^{\ell} \big\|\upsilon(t)\big\|^2_Udt<\infty.$ 

We assume that the linear operator $-A(t)$ fulfills the following conditions:
\begin{itemize}
	\item[(B1)] For every $\delta$ with $Re(\delta) \leq0$, the operator $\delta I+A(t)$ is invertible in $L(\mathrm{X})$ and 
	\begin{eqnarray*}
		\left\|\left[\delta I+A(t)\right]^{-1}\right\|\leq \frac{C}{|\delta|+1},
	\end{eqnarray*}
	where $C$ is a positive constant.	
	\item[(B2)] For each $s, \tau, t \in I$, there exists a constant $\gamma\in (0,1]$ such that
	\begin{eqnarray*}
		\left\|\left[A(s)-A(\tau)\right]A(t)^{-1}\right\|\leq C|s-\tau|^{\gamma},	
	\end{eqnarray*}
	where the constants $\gamma$ and $C>0$ are independent of both $s, \tau$ and $t$.
\end{itemize}
\begin{remark}\label{r1}
	From $(B1)$, if we substitute $\delta=0$ and $t=0$, then $\left\|A^{-1}(0)\right\|\leq C$, where $C$ is a constant defined as in $(B1)$.
\end{remark}
\begin{lemma}\cite{M2001}\label{l1}
	For any $\vartheta_\ell\in L^p(\Upsilon,X)$ and $p\geq2,$  there exists a function $\kappa \in  L_{\Upsilon}^p([0,\ell],L_2^0)$ such that 
	\begin{eqnarray*}
		\vartheta_\ell= E\vartheta_\ell  + \int_{0}^{\ell}\kappa(s)dw(s).
	\end{eqnarray*}
\end{lemma}	
\begin{lemma}\label{l2} \cite{k} Let  $\mathcal{V}$ be ${L_2^0}$-valued predictable process with $p \geq 2$ such that $E\left(\displaystyle\int_{0}^{\ell} \|\mathcal{V}(s)\|_{L_2^0}^pds\right)< \infty,$ then
	\begin{eqnarray*}E\left(\sup\limits_{s\in[0,t]} \left\| \int_{0}^{s}\mathcal{V}(\xi)dv(\xi) \right \|^p \right)&\leq& c_p\sup\limits_{s\in[0,t]}E\left( \left\|\int_{0}^{s}  \mathcal{V}(\xi)dv(\xi)\right\|^p\right)\\&\leq& C_pE\left( \int_{0}^{t}  \|\mathcal{V}(\xi)\|^p_{L_2^0}d\xi\right),~t\in[0,\ell],
	\end{eqnarray*}
	where $c_p=\left(\dfrac{p}{p-1}\right)^p$ and $C_p=\left(\dfrac{p}{2}(p-1)\right)^{\frac{p}{2}}\left(\dfrac{p}{p-1}\right)^{\frac{p^2}{2}}.$
\end{lemma}
\begin{definition}\cite{M2004}\label{d1}
	The stochastic process $\vartheta: (-\infty,\ell]\rightarrow{X}$ is said to be a mild solution of the problem (\ref{1.1}) if $\vartheta_0={A}^{-1}(0)\phi\in \wp$ satisfying $\|{A}^{-1}(0)\phi\|_{\wp}^2<\infty$, the restriction $\vartheta(\cdot)$ to the interval $[0,\ell)$ is a continuous and satisfy the following integral equation:	
	\begin{eqnarray} 
		\vartheta(t)&=&  {A}^{-1}(0)\phi(0)-f(0,{A}^{-1}(0)\phi(0))+f(t,\vartheta_t)+\displaystyle\int_{0}^{t}\varPsi(t-\varrho,\varrho)\mathcal{U}(\varrho)\phi(0)d\varrho\nonumber\\
		&&+\displaystyle\int_{0}^{t}\varPsi(t-\varrho,\varrho)\big[{B}\upsilon(\varrho)+g\left(\varrho,\vartheta_{\varrho},\vartheta(\varrho)\right)\big]d\varrho +\displaystyle\int_{0}^{t}\varPsi(t-\varrho,\varrho)h(\varrho,\vartheta(\varrho))dw(\varrho)\nonumber  \\
		&&+\displaystyle\int_{0}^{t}  \int_{0}^{\varrho}\varPsi(t-\varrho,\varrho)\varPhi(\varrho,s)\big[{B}\upsilon(s)+g\left(s,\vartheta_s,\vartheta(s)\right)\big]dsd\varrho\nonumber\\
		&&+\displaystyle\int_{0}^{t}  \int_{0}^{\varrho}\varPsi(t-\varrho,\varrho)\varPhi(\varrho,s)h(s,\vartheta(s))dw(s)d\varrho,
	\end{eqnarray}
\end{definition}
where $\varPsi(t,\varrho),~ \varPhi(t,\varrho)$ and $\mathcal{U}(t)$ are given by the following equations \cite{R}
\begin{eqnarray}
	\varPsi(t,\varrho)=\gamma \int_{0}^{\infty}\lambda t^{\gamma-1} \xi_{\gamma}(\lambda)e^{-t^\gamma \lambda {A}( \varrho )}d\lambda,
\end{eqnarray}
\begin{eqnarray}
	\varPhi(t,\varrho)=\sum \limits_{k=1}^{\infty}\varPhi_k(t,\varrho),
\end{eqnarray} and
\begin{eqnarray}
	\mathcal{U}(t)=-{A}(t){A}^{-1}(0)-\int_{0}^{t}\varPhi(t,\varrho){A}(\varrho){A}^{-1}(0)d\varrho,
\end{eqnarray}
$\xi_{\gamma}$ denotes the probability density function on $[0,\infty)$ and it's Laplace transform is given by
\begin{eqnarray*}
	\int_{0}^{\infty}e^{-\varrho z} \xi_{\gamma}(\varrho)d\varrho=\sum \limits_{j=0}^{\infty}\frac{(-z)^j}{\Gamma(1+\gamma^j)}, \quad 0<\gamma \leq 1, \quad z>0,
\end{eqnarray*} where $\Gamma(\cdot)$ is the Euler-Gamma function and 
\begin{eqnarray*}
	\varPhi_1(t,\varrho)&=&[{A}(t)-{A}(\varrho)]\varPsi(t-\varrho,\varrho), \\
	\varPhi_{k+1}(t,\varrho)&=&\int_{\varrho}^{t}\varPhi_k(t,s)\varPhi_1(s,\varrho)d\varrho, \quad k=1,2,3, \ldots.
\end{eqnarray*}
\begin{lemma} \cite{M2004}\label{l3}
	The operators $\varPsi(t-\varrho,\varrho)$ and ${A}(t)\varPsi(t-\varrho,\varrho)$ are continuous in the uniform operator topology in variables $t$ and $\varrho,$ where $t\in {J}$ and $0\leq \varrho \leq t-\epsilon$  for each $\epsilon >0$ and satisfy the following inequalities:
	\begin{eqnarray}
		\|\varPsi(t-\varrho,\varrho)\| &\leq& C(t-\varrho)^{\gamma-1},
	\end{eqnarray}
	where $C>0$ is a constant independent of both $t$ and $\varrho.$ Furthermore,
	\begin{eqnarray}
		\|\varPhi(t,\varrho)\| &\leq& C(t-\varrho)^{\beta-1},
	\end{eqnarray}
	and
	\begin{eqnarray}
		\|\mathcal{U}(t)\| &\leq& C(1+t^{\beta}),
	\end{eqnarray}
	where $C$ is a constant.
\end{lemma}
We obtain the following result by Lemma \ref{l3}.
\begin{lemma}\cite{C2018}\label{l4}
	For any $t\in {J}$, the integral $\displaystyle\int_{0}^{t}\varPsi(t-\varrho,\varrho)\mathcal{U}(\varrho)d\varrho$ is uniformly continuous in the operator norm $L({X})$ and	
	\begin{eqnarray*}
		\left\|\int_{0}^{t}\varPsi(t-\varrho,\varrho)\mathcal{U}(\varrho)d\varrho\right\|\leq C^2 t^{\gamma}\left(\frac{1}{\gamma}+t^{\beta}\mathbf{B}(\gamma,\beta+1)\right),
	\end{eqnarray*}
	where 
	\begin{eqnarray*}
		\mathbf{B}(\beta,\gamma)=\int_{0}^{1}t^{\beta-1}(1-t)^{\gamma-1}dt
	\end{eqnarray*}
	is the Beta function.
\end{lemma}
\begin{lemma}\cite{C2018}\label{l5}
	For each $t\in {J}$ and $u\in L^1[0,\ell]$, we have
	\begin{eqnarray*}\int_{0}^{t}\int_{0}^{\varrho}(t-\varrho)^{\gamma-1}(\varrho-s)^{\beta-1}u(s)dsd\varrho=\mathbf{B}(\beta,\gamma)\int_{0}^{t}(t-\varrho)^{\beta+\gamma-1}u(\varrho)d\varrho.
	\end{eqnarray*}
\end{lemma}
\begin{definition}\cite{AH2006} \label{d2}
	The fractional integral of a function $y \in L^1([0,\infty),\mathbb{R})$ is defined as
	\begin{eqnarray*}
		I^{\gamma}y(t)=\frac{1}{\Gamma(\gamma)}\int_{0}^{t}(t-s)^{\gamma-1}y(s)ds,
	\end{eqnarray*}
	where $\gamma>0$.
\end{definition}
\begin{definition}\cite{AH2006}\label{d3}
	The Caputo fractional derivative of a function $y \in L^1([0,\infty),\mathbb{R})$ which  can be
	defined as
	\begin{eqnarray*}
		^CD^{\gamma}y(t)=\frac{1}{\Gamma(n-\gamma)}\int_{0}^{t}(t-s)^{n-\gamma-1}y^{(n)}(s)ds=I^{n-\gamma}_ty(t),
	\end{eqnarray*}	
	where $n-1<\gamma<n,~n \in\mathbb{N}$ and $y$ is at least $n$-times differentiable .	
\end{definition}
\begin{definition}\cite{J1978} \label{d4}
	The Kuratowski measure of non-compactness $\mu(\cdot)$ defined on bounded set $\mathcal{S}$ of a Banach space $X$ is 
	\begin{eqnarray*}
		\mu(\mathcal{S})=\inf\left\{\delta>0: \mathcal{S}=\bigcup\limits_{k=1}^{\infty}\mathcal{S}_k \mbox{ and } diam(\mathcal{S}_k)\leq\delta \quad\mbox{ for } \quad k=1,2,\dots,n\right\}.
	\end{eqnarray*}
\end{definition}
\begin{lemma}\cite{K1985, PXY2019}\label{l6}
	Let $\mathcal{X}$ be a Banach space and $M, N\subset \mathcal{X}$ be bounded. The following properties are satisfied:
	\begin{itemize}
		\item[(1)] $\mu(M)\leq\mu(N)$ if $M\subset N$;
		\item[(2)] $\mu(M\cup N)=\max\{\mu(M),\mu(N) \}$;
		\item[(3)]  $\mu(M+N)\leq\mu(M)+\mu(N)$, where $M+N=\{x=y+z: y\in M, z\in N\}$;
		\item[(4)] $\mu(\lambda M)=|\lambda|\mu(M)$, where $\lambda$ is a real number;
		\item[(5)] For any $x\in \mathcal{X}$, $\mu(M+x)=\mu(M)$;
		\item[(6)] $\mu(M)=0$ if and only if $\overline{M}$ is compact, where $\overline{M}$ means the closure hull of $M$;
		\item[(7)] $\mu(M)=\mu(\overline{M})=\mu(conv M)$, where $conv M$ means the convex hull of $M$;
		\item[(8)] If the function $L:D(L)\subset X\rightarrow H$ is a Lipschitz continuous, then $\mu(L(S))\leq k\mu(M)$, where $M\subset D(L)$ is a bounded subset, $H$ is another Banach space and $k$ is Lipschitz constant.
	\end{itemize}
\end{lemma}
In this article, $\mu(\cdot)$ and $\mu_C(\cdot)$ denote the Kuratowski measure of noncompactness on the bounded set of $\mathcal{X}$ and $C(J ,\mathcal{X})$ respectively. For any subset $\mathcal{P}$ of $C(J , \mathcal{X})$, define the set $\mathcal{P}(t) = \{\vartheta(t)~|~\vartheta \in \mathcal{P},~t\in J\}\subset \mathcal{X}$. If the set $\mathcal{P}$ is bounded in $C(J , \mathcal{X})$, then
$\mathcal{P}(t)$ is bounded in $\mathcal{X}$ and $\mu(\mathcal{P}(t)) \leq \mu_C(\mathcal{P})$. We refer to the monographs \cite{J1978} and
\cite{K1985} for more details about the Kuratowski measure of noncompactness.
\begin{lemma}\cite{J1978} \label{l7}
	Let $\mathcal{X}$ be a Banach space, and $\mathcal{P}$ be the subset of $C(J ,\mathcal{X})$ which is bounded and equicontinuous. Then $\mu(\mathcal{P}(t))$ is continuous on $J$, and $\mu_C(\mathcal{P}) = \max \limits_{t\in J}\mu
	(\mathcal{P}(t))$.
\end{lemma}
\begin{lemma}\cite{PXY2019}\label{l8}
	Let $\mathcal{P}$ be the bounded subset of a Banach space $\mathcal{X}$. Then there exists a countable set $\mathcal{P}_0 \subset \mathcal{P}$, such that $\mu(\mathcal{P}) \leq 2\mu(\mathcal{P}_0)$.
\end{lemma}
\begin{lemma}\cite{H1983} \label{l9}
	Let $\mathcal{X}$ be a Banach space. If  $\mathcal{Q}=\left\{\vartheta^n\right\}_{n=1}^{\infty}\subset C(J ,\mathcal{X})$ is a countable set and there exists a function $\alpha\in L^1(J,\mathbb{R}^+)$ such that for each $n \in \mathbb{N}$
	\begin{eqnarray*}\left\|\vartheta^n\right\|\leq \alpha(t), ~a.e.~ t\in J. 
	\end{eqnarray*}
	Then $\mu(\mathcal{Q}(t))$ is Lebesgue integrable on $J$, and
	\begin{eqnarray*}
		\mu\left(\left\{\int_{0}^{\ell}\vartheta^n(\varrho)d\varrho: n\in \mathbb{N}\right\}\right)\leq 2\int_{0}^{\ell}\mu(\mathcal{Q}(\varrho))d\varrho.
	\end{eqnarray*}
\end{lemma}
We need the following Lemma to deal with the measure of stochastic integral term.
\begin{lemma}\cite{L2019} \label{l10}
	If $W \subset L^p(J;L^0_2
	(Y,X))$ and $w(t)$ is a Q-Wiener process. For
	every $p\geq 2$, the Hausdorff measure of non-compactness $\mu$ satisfies
	\begin{eqnarray*}
		\mu\left(\int_{0}^{\ell}W(\varrho)dw(\varrho)\right)\leq \sqrt{\ell\dfrac{p}{2}(p-1)}\mu(W(\varrho)),
	\end{eqnarray*}
	where 
	\begin{eqnarray*}
		\int_{0}^{\ell}W(\varrho)dw(\varrho)=\left\{\int_{0}^{\ell}\vartheta(\varrho)dw(\varrho) \mbox{ for all } \vartheta \in W,~ t\in J\right\}.
	\end{eqnarray*}
\end{lemma}
\begin{remark}
	When $p=2$ in Lemma \ref{l10}, then 
	\begin{eqnarray*}
		\mu\left(\int_{0}^{\ell}W(\varrho)dw(\varrho)\right)\leq \sqrt{\ell TrQ}\mu(W(\varrho)).
	\end{eqnarray*}
\end{remark}
\begin{definition}\cite{J1978}\label{d5}
	Let $\mathcal{Q}$ be a nonempty subset of a Banach space $\mathcal{X}$. A continuous mapping $\sigma:\mathcal{Q}\rightarrow \mathcal{X}$ is said to be $k$-set-contractive if there exists a constant $k\in[0, 1)$ such that, for every bounded set $\mathcal{P}\subset \mathcal{Q}$,
	\begin{eqnarray*}\mu(\sigma(\mathcal{P})) \leq k\mu(\mathcal{P}).
	\end{eqnarray*}
\end{definition}
\begin{lemma}\cite{K1985} \label{l11}
	Let $\mathcal{Q}$ be a bounded closed and convex subset of a Banach space $\mathcal{X}$, the operator $\sigma:\mathcal{Q}\rightarrow \mathcal{Q}$ is $k$-set-contractive. Then $\sigma$ has at
	least one fixed point in $\mathcal{Q}$.
\end{lemma}
\begin{definition}\label{d6} \cite{c2}
	The control problem (\ref{1.1}) is called controllable on $J$ if for each continuous stochastic process $A^{-1}(0)\phi\in\wp$ defined on $J_0$, there exists a  control function $\upsilon \in {L^2(J,U)}$ such that the mild solution $\vartheta$ of (\ref{1.1}) satisfies $\vartheta(\ell)=\vartheta _1,$ where $\vartheta_1$ and $\ell$ are predefined final state and time respectively.
\end{definition}
\section{\textbf{Main Result}}
\begin{itemize} 
	\item[(C1)] The linear operator $B$ is bounded, i.e., 
	$\|B\|\leq m_1$ for some constant $m_1>0$.	
	
	\item[(C2)] The linear operator $Q: L_{\Upsilon}^2(J,U) \rightarrow L^2(\Upsilon,X)$ defined by
	\begin{eqnarray*} 
		(Q\upsilon)(t)=\displaystyle\int_{0}^{t}\varPsi(t-\varrho,\varrho){B}\upsilon(\varrho)d\varrho+\int_{0}^{t}  \int_{0}^{\varrho}\varPsi(t-\varrho,\varrho)\varPhi(\varrho,s){B}\upsilon(s)dsd\varrho, 
	\end{eqnarray*}
	has an invertible operator $Q^{-1}$ and there exists a positive constant $m_2$ such that $\big\|Q^{-1}\big\|\leq m_2.$
	
	\item[(C3)] The continuous function $f: J\times \wp \rightarrow \mathrm{X}$ satisfies the following conditions:
	\begin{itemize}
		\item [(i)] For each $t_1, t_2 \in  J$ and $\vartheta_t, \vartheta_s \in \wp,$  there exists a constant $L_f>0$ such that
		\begin{eqnarray*}
			\|f(t_1,\vartheta_t)-f(t_2,\vartheta_s)\|^2\leq L_f(|t_1-t_2|+\|\vartheta_t-\vartheta_s\|^2_{\wp}).	\end{eqnarray*}
		\item [(ii)] For each $t\in J$ and $\vartheta_t \in \wp,$ there exists a constant $\tilde{L}_f>0$ such that
		\begin{eqnarray*}
			\|f(t,\vartheta_t)\|^2\leq \tilde{L}_f(1+\|\vartheta_t\|^2).
		\end{eqnarray*}
		\item [(iii)]For every bounded and countable subset $\mathcal{P}$ of $X$, there exists a positive constant $\hat{L}_f$ such that
		\begin{eqnarray*}
			\mu( f (t, \mathcal{P})) \leq \hat{L}_f\mu(\mathcal{P}),~t\in J.
		\end{eqnarray*}
	\end{itemize}
	\item[(C4)] The continuous function $g:{J} \times\wp \times X \rightarrow X$ fulfills the following conditions:
	\begin{itemize}	
		\item[(i)] $g(\cdot,\vartheta_t,\vartheta)$ is  strongly measurable for each $\vartheta_t\in \wp,~\vartheta \in \mathrm{X}.$  
		\item[(ii)] For each $t_1, t_2 \in  J$, $\vartheta_t, \vartheta_s \in \wp$ and $\vartheta,\tilde{\vartheta}\in X$, there exists a constant $L_g>0$ such that
		\begin{eqnarray*}\big\|g\big(t_1,\vartheta_t,\vartheta(t)\big)-g\big(t_2,\vartheta_s,\tilde{\vartheta}(t)\big)\big\|^2\leq L_g\left(|t_1-t_2|+\|\vartheta_t-\vartheta_s\|^2_{\wp}+\|\vartheta(t)-\tilde{\vartheta}(t)\|^2\right).\hspace{-2cm}
		\end{eqnarray*}
		\item[(iii)] For each $r>0$, there exists a non-zero function $\mathcal{G}_r \in {L}^{\frac{1}{2\gamma_1}}(J,\mathbb{R}^+)$ such that
		\begin{eqnarray*}
			\sup\limits_{\|\vartheta_t\|^2,\|\vartheta\|^2\leq r}\big\|g\big(t,\vartheta_t,\vartheta(t)\big)\big\|^2\leq\mathcal{G}_r(t), ~ \mbox{for a.e. }t\in J,
		\end{eqnarray*}
		and
		\begin{eqnarray*}\lim \limits_{r\rightarrow\infty}\inf\frac{\|\mathcal{G}_r\|_{{L}^{\frac{1}{2\gamma_1}}[0,\ell]}}{r}=\delta<\infty,
		\end{eqnarray*}
		where $2\gamma_1 \leq \min\{\gamma, \beta \}.$
		\item[(iv)]For any bounded and countable subsets $\mathcal{P},\mathcal{Q}$ of $X$, there exists a positive constant $\hat{L}_g$ such that 
		\begin{eqnarray*}
			\mu( g (t,\mathcal{P}, \mathcal{Q})) \leq \hat{L}_g\max\{\mu(\mathcal{P}),\mu(\mathcal{Q})\}, ~t\in J.
		\end{eqnarray*}
	\end{itemize}
	
	\item [(C5)] The continuous function $ h:J\times \mathrm{X}\rightarrow  L_2^0$ satisfies the following conditions:\\ 
	\begin{itemize}
		\item [(i)] For each $t_1, t_2 \in  J$, and $\vartheta, \tilde{\vartheta} \in X,$  there exists a constant $L_{h}>0$ such that
		\begin{eqnarray*}
			\big\|h(t,\vartheta(t))-h(t,\tilde{\vartheta}(t))\big\|^2_{Q}\leq L_{h}\|\vartheta(t)-\tilde \vartheta(t)\|^2.
		\end{eqnarray*}
		\item [(ii)]  For each $t\in J$ and $\vartheta_t \in \wp,$ there exists a constant $\tilde{L}_h>0$ such that
		\begin{eqnarray*}\big\|h(t,\vartheta(t))\big\|_{Q}^2\leq\tilde{L}_{h}(1+\|\vartheta(t)\|^2).
		\end{eqnarray*}
		\item [(iii)] For any bounded and countable subset $\mathcal{P}$ of $X$, there exists a positive constant $\hat{L}_h$ such that 
		\begin{eqnarray*}
			\mu( h(t, \mathcal{P})) \leq \hat{L}_h\mu(\mathcal{P}), ~t\in J.
		\end{eqnarray*}
	\end{itemize}
\end{itemize}
For simplicity, we take the following notations
\begin{eqnarray*}
	\lambda_1=\left(\dfrac{1}{\gamma}+{\ell}^{\beta}\mathbf{B}(\gamma,\beta+1)\right),~ \lambda_2=\left(\left(\frac{1-\gamma_1}{\gamma-\gamma_1}\right)^{2(1-\gamma)}+C^2{\ell}^{2\beta}\left(\mathbf{B}(\beta,\gamma)\right)^2\left(\frac{1-\gamma_1}{\beta+\gamma-\gamma_1}\right)^{2(1-\gamma_1)}\right),
\end{eqnarray*} 
\begin{eqnarray*}\lambda_3=\left(\frac{1}{{\gamma}^2}+C^2\left(\mathbf{B}(\gamma,\beta)\right)^2\frac{\ell^{2\beta}}{(\beta+\gamma)^2}\right),~ 	\lambda_4=\left(\frac{1}{{2\gamma-1}}+\ell^{2\beta}\frac{C^2\left(\mathbf{B}(\gamma,\beta)\right)^2}{(2\beta+2\gamma-1)}\right),\hspace{2.5cm}
\end{eqnarray*}
\begin{eqnarray*}
	\lambda_5=\left(\dfrac{1}{\gamma}+2C{\ell}^{\beta}\mathbf{B}(\gamma,\beta)\dfrac{{\ell}^{2\beta}}{\beta+\gamma}\right),~ 	l_2=10m_2^2C^2{\ell}^{2(\gamma-\gamma_1)}\lambda_2,~ l_3=10m_2^2C^2TrQ\tilde{L}_h{\ell}^{2\gamma}\lambda_3, \hspace{2cm}
\end{eqnarray*}
\begin{eqnarray*}
	l_1=10m_2^2\bigg[E\|\vartheta_1\|^2+TrQ\int_{0}^{\ell}E\|\kappa(\varrho)\|^2d\varrho+C^2\left[1+\tilde{L}_f+{\ell}^{2\gamma}\lambda_1^2\right]E\|\phi(0)\|^2+\tilde{L}_f(2+E\|\vartheta_{\ell}\|^2)\bigg].
\end{eqnarray*}
Let $\mathcal{Z}=\{\vartheta: \vartheta\in\ C((-\infty,\ell],X),\vartheta(0)={A}^{-1}(0)\phi(0)\}$ be the space of continuous functions. For $r>0$, we define the set $\Theta_r=\{\vartheta\in \mathcal{Z}: E\|\vartheta(t)\|^2\leq r \}$. For any $\vartheta\in \Theta_r$ and $0\leq t\leq\ell$,
\begin{eqnarray*}
	\|\vartheta_t\|_{\wp}\leq \sup\limits_{-\infty<s\leq0}\|\vartheta_t(s)\|\leq\sup\limits_{-\infty<s\leq\ell}\|\vartheta(s)\|\leq r.
\end{eqnarray*} 

\begin{lemma} \label{l12}
	There exist constants $l_1,l_2, l_3>0$ such that for every $\vartheta \in C(J,X)$ satisfying (\ref{1.1}), 
	\begin{eqnarray*}
		E\|\upsilon(t)\|^2 \leq	l_1+l_2\|\mathcal{G}_r\|_{{L}^{\frac{1}{2\gamma_1}}[0,\ell]}+l_3(1+E\|\vartheta\|^2).
	\end{eqnarray*}
\end{lemma}
\begin{proof}
	In view of assumption (C2) for an arbitrary function $\vartheta(\cdot),$ the control function is defined as follows:
	\begin{eqnarray}\label{3.1}
		\upsilon(t)&=&Q^{-1}\Bigg[E\vartheta_1  + \int_{0}^{t}\kappa(s)dw(s) -{A}^{-1}(0)\phi(0)+f(0,{A}^{-1}(0)\phi(0))-f(t,\vartheta_{t})\nonumber \\
		&&\hspace{1cm}-\int_{0}^{t}\varPsi(t-\varrho,\varrho) \mathcal{U}(\varrho)\phi(0)d\varrho - \int_{0}^{t} \varPsi(t-\varrho,\varrho)g\left(\varrho,\vartheta_{\varrho},\vartheta(\varrho)\right)d\varrho \nonumber \\
		&&\hspace{1cm}-\int_{0}^{t} \varPsi(t-\varrho,\varrho)h(\varrho,\vartheta(\varrho))dw(\varrho)-\int_{0}^{t}  \int_{0}^{\varrho}\varPsi(t-\varrho,\varrho)\varPhi(\varrho,s)g\left(s,\vartheta_s,\vartheta(s)\right)dsd\varrho\nonumber \\
		&&\hspace{1cm}-\int_{0}^{t}  \int_{0}^{\varrho}\varPsi(t-\varrho,\varrho)\varPhi(\varrho,s)h(s,\vartheta(s))dw(s)d\varrho \Bigg].  
	\end{eqnarray} 
	\begin{eqnarray}\label{3.2}
		E\|\upsilon(t)\|^2&=&10\left\|Q^{-1}\right\|^2\Bigg[E\|\vartheta_1\|^2  +E\left\|\int_{0}^{\ell}\kappa(s)dw(s)\right\|^2+E\left\|{A}^{-1}(0)\phi(0)\right\|^2\nonumber\\
		&&\quad+E\left\|f(0,{A}^{-1}(0)\phi(0))\right\|^2+E\big\|f(\ell,\vartheta_{\ell})\big\|^2+E\left\|\int_{0}^{\ell}\varPsi(\ell-\varrho,\varrho) \mathcal{U}(\varrho)\phi(0)d\varrho\right\|^2 \nonumber \\
		&&\quad+E\left\|\int_{0}^{\ell} \varPsi(\ell-\varrho,\varrho)g\left(\varrho,\vartheta_{\varrho},\vartheta(\varrho)\right)d\varrho\right\|^2+E\left\|\int_{0}^{\ell} \varPsi(\ell-\varrho,\varrho)h(\varrho,\vartheta(\varrho))dw(\varrho)\right\|^2 \nonumber\\
		&&\quad+E\left\|\int_{0}^{\ell}  \int_{0}^{\varrho}\varPsi(\ell-\varrho,\varrho)\varPhi(\varrho,s)g\left(s,\vartheta_s,\vartheta(s)\right)dsd\varrho\right\|^2\nonumber  \\
		&&\quad+E\left\|\int_{0}^{\ell}  \int_{0}^{\varrho}\varPsi(\ell-\varrho,\varrho)\varPhi(\varrho,s)h(s,\vartheta(s))dw(s)d\varrho\right\|^2 \Bigg].  
	\end{eqnarray} 
	Combining with Remark~\ref{r1}, Lemma \ref{l2}, and assumptions (C2)-(C3) in (\ref{3.2}), we get
	\begin{eqnarray}\label{3.3}
		E\|\upsilon(t)\|^2
		&\leq&10m^2_2\Bigg[E\|\vartheta_1\|^2+Tr(Q)\int_{0}^{\ell}E\left\|\kappa(s)\right\|^2ds+C^2E\big\|\phi(0)\big\|^2+\tilde{L}_f\left(1+C^2E\left\|\phi(0)\right\|^2\right)\nonumber\\
		&&\quad+\tilde{L}_f\left(1+E\left\|\vartheta_{\ell}\right\|^2\right)+\left(\int_{0}^{\ell}E\|\varPsi(\ell-\varrho,\varrho) \mathcal{U}(\varrho)\phi(0)\|d\varrho\right)^2\nonumber \\
		&&\quad+\left(\int_{0}^{\ell}E\left\| \varPsi(\ell-\varrho,\varrho)g\left(\varrho,\vartheta_{\varrho},\vartheta(\varrho)\right)\right\|d\varrho\right)^2\nonumber\\
		&&\quad+Tr(Q)\left(\int_{0}^{\ell} E\left\|\varPsi(\ell-\varrho,\varrho)h(\varrho,\vartheta(\varrho))\right\|d\varrho\right)^2\nonumber\\
		&&\quad+\left(\int_{0}^{\ell}  \int_{0}^{\varrho}E\|\varPsi(\ell-\varrho,\varrho)\varPhi(\varrho,s)g\left(s,\vartheta_s,\vartheta(s)\right)\|dsd\varrho\right)^2\nonumber  \\
		&&\quad +Tr(Q)\left(\int_{0}^{\ell}  \int_{0}^{\varrho}E\|\varPsi(\ell-\varrho,\varrho)\varPhi(\varrho,s)h\left(s,\vartheta(s)\right)\|dsd\varrho\right)^2\Bigg].
	\end{eqnarray}
	Using the Lemmas~\ref{l3},~\ref{l5} and assumptions (C4)-(C5) in (\ref{3.3}), we get
	\begin{eqnarray}\label{3.4}
		E\|\upsilon(t)\|^2&\leq&10m^2_2\Bigg[E\|\vartheta_1\|^2+Tr(Q)\int_{0}^{\ell}E\left\|\kappa(s)\right\|^2ds+C^2(1+\tilde{L}_f)E\big\|\phi(0)\big\|^2\nonumber\\
		&&\quad+\tilde{L}_f\left(2+E\left\|\vartheta_{\ell}\right\|^2\right)+C^4E\left\|\phi(0)\right\|^2\left(\int_{0}^{\ell}(\ell-\varrho)^{\gamma-1}(1+\varrho^{\beta})d\varrho\right)^2 \nonumber\\
		&&\quad+C^2\left(\int_{0}^{\ell}(\ell-\varrho)^{\gamma-1}\left(E\left\|g\left(\varrho,\vartheta_{\varrho},\vartheta(\varrho)\right)\right\|^2\right)^{\frac{1}{2}}d\varrho\right)^{2}\nonumber\\
		&&\quad+C^2Tr(Q)\left(\int_{0}^{\ell}(\ell-\varrho)^{\gamma-1}E\left\|h\left(\varrho,\vartheta(\varrho)\right)\right\|_Qd\varrho\right)^2\nonumber\\
		&&\quad+C^4\left(\int_{0}^{\ell}\int_{0}^{\varrho}(\ell-\varrho)^{\gamma-1}(\varrho-s)^{\beta-1}\left(E\left\|g\left(s,\vartheta_{s},\vartheta(s)\right)\right\|^2\right)^{\frac{1}{2}}dsd\varrho \right)^2\nonumber\\
		&&\quad+C^4Tr(Q)\left(\int_{0}^{\ell}\int_{0}^{\varrho}(\ell-\varrho)^{\gamma-1}(\varrho-s)^{\beta-1}E\|h(s,\vartheta(s))\|_Qdsd\varrho \right)^2\Bigg]\nonumber\\
		&\leq&10m^2_2\Bigg[E\|\vartheta_1\|^2+Tr(Q)\int_{0}^{\ell}E\left\|\kappa(s)\right\|^2ds+C^2(1+\tilde{L}_f)E\big\|\phi(0)\big\|^2\nonumber\\
		&&\quad+\tilde{L}_f\left(2+E\left\|\vartheta_{\ell}\right\|^2\right)+C^4E\left\|\phi(0)\right\|^2\left(\int_{0}^{\ell}(\ell-\varrho)^{\gamma-1}(1+\varrho^{\beta})d\varrho\right)^2\nonumber \\
		&&\quad+C^2\left(\int_{0}^{\ell}(\ell-\varrho)^{\gamma-1}\left(\mathcal{G}_r(\varrho)\right)^{\frac{1}{2}}d\varrho\right)^{2}\nonumber\\
		&&\quad+C^2Tr(Q)\left(\int_{0}^{\ell}(\ell-\varrho)^{\frac{\gamma-1}{2}} \cdot(\ell-\varrho)^{\frac{\gamma-1}{2}}E\left\|h\left(\varrho,\vartheta(\varrho)\right)\right\|_Qd\varrho\right)^2\nonumber\\
		&&+C^4{(\mathbf{B}(\gamma,\beta))}^2\left(\int_{0}^{\ell}(\ell-\varrho)^{\beta+\gamma-1}\left(\mathcal{G}_r(\varrho)\right)^{\frac{1}{2}}d\varrho \right)^2\nonumber\\
		&&+C^4Tr(Q){(\mathbf{B}(\gamma,\beta))}^2\left(\int_{0}^{\ell}(\ell-\varrho)^{\frac{\beta+\gamma-1}{2}} \cdot(\ell-\varrho)^{\frac{\beta+\gamma-1}{2}} E\|h(\varrho,\vartheta(\varrho))\|d\varrho \right)^2\Bigg].\hspace{0.5cm}
	\end{eqnarray}
	Using H\"{o}lder inequality, Cauchy-Schwarz inequality and Lemma~\ref{l4}~in (\ref{3.4}), we get
	\begin{eqnarray*}		E\|\upsilon(t)\|^2&\leq&10m^2_2\Bigg[E\|\vartheta_1\|^2+Tr(Q)\int_{0}^{\ell}E\left\|\kappa(s)\right\|^2ds+C^2(1+\tilde{L}_f)E\big\|\phi(0)\big\|^2\\
		&&\quad+\tilde{L}_f\left(2+E\left\|\vartheta_{\ell}\right\|^2\right)+C^4E\left\|\phi(0)\right\|^2\ell^{2\gamma}\left( \frac{1}{\gamma}+\ell^{\beta}\mathbf{B}(\gamma,\beta+1)\right)^2\\
		&&\quad+\left( \int_{0}^{\ell} (\ell-\varrho)^{\frac{\gamma-1}{1-\gamma_1}}d\varrho \right)^{2(1-\gamma_1)} \left(\int_{0}^{\ell}\mathcal{G}^{\frac{1}{2\gamma_1}}_r(\varrho)d\varrho\right)^{2\gamma_1}\\
		&&\quad+C^2Tr(Q)\left(\int_{0}^{\ell}(\ell-\varrho)^{\gamma-1}d\varrho\right)\left(\int_{0}^{\ell}(\ell-\varrho)^{\gamma-1}E\left\|h\left(\varrho,\vartheta(\varrho)\right)\right\|_Q^2d\varrho\right)\\
		&&\quad+C^4{(\mathbf{B}(\gamma,\beta))}^2\left( \int_{0}^{\ell} (\ell-\varrho)^{\frac{\beta+\gamma-1}{1-\gamma_1}}d\varrho \right)^{2(1-\gamma_1)} \left(\int_{0}^{\ell}\mathcal{G}^{\frac{1}{2\gamma_1}}_r(\varrho)d\varrho\right)^{2\gamma_1}\\
		&&\quad+C^4Tr(Q){(\mathbf{B}(\gamma,\beta))}^2\left(\int_{0}^{\ell}(\ell-\varrho)^{\beta+\gamma-1}d\varrho\right)\times\\
		&&\quad\hspace{5cm}\left(\int_{0}^{\ell}(\ell-\varrho)^{\beta+\gamma-1}E\left\|h\left(\varrho,\vartheta(\varrho)\right)\right\|_Q^2d\varrho\right)\Bigg]\\
		&\leq& 10m^2_2\Bigg[E\|\vartheta_1\|^2+Tr(Q)\int_{0}^{\ell}E\left\|\kappa(s)\right\|^2ds+C^2(1+\tilde{L}_f)E\big\|\phi(0)\big\|^2\\
		&&\quad+\tilde{L}_f\left(2+E\left\|\vartheta_{\ell}\right\|^2\right)+C^4E\left\|\phi(0)\right\|^2\ell^{2\gamma}\left(\frac{1}{\gamma}+\ell^{\beta}\mathbf{B}(\gamma,\beta+1)\right)^2\\
		&&\quad+C^2\left(\frac{1-\gamma_1}{\gamma-\gamma_1}\right)^{2(1-\gamma_1)}{\ell}^{2(\gamma-\gamma_1)}\|\mathcal{G}_r\|_{{L}^{\frac{1}{2\gamma_1}}[0,\ell]} +C^2Tr(Q)\tilde{L}_h\frac{\ell^{2\gamma}}{\gamma^2}(1+E\|\vartheta\|^2_Z)\\
		&&\quad+C^4{(\mathbf{B}(\gamma,\beta))}^2\left(\frac{1-\gamma_1}{\beta+\gamma-\gamma_1}\right)^{2(1-\gamma_1)}{\ell}^{2(\beta+\gamma-\gamma_1)}\|\mathcal{G}_r\|_{{L}^{\frac{1}{2\gamma_1}}[0,\ell]}\\
		&&\quad+C^4Tr(Q){(\mathbf{B}(\gamma,\beta))}^2\tilde{L}_h\frac{\ell^{2(\beta+\gamma)}}{(\beta+\gamma)^2}(1+E\|\vartheta\|^2_Z)\bigg]\\
		&=& 10m^2_2\Bigg[E\|\vartheta_1\|^2+Tr(Q)\int_{0}^{\ell}E\left\|\kappa(s)\right\|^2ds+C^2(1+\tilde{L}_f)E\big\|\phi(0)\big\|^2\\
		&&\quad+\tilde{L}_f\left(2+E\left\|\vartheta_{\ell}\right\|^2\right)+C^4E\left\|\phi(0)\right\|^2\ell^{2\gamma}\left(\frac{1}{\gamma}+\ell^{\beta}\mathbf{B}(\gamma,\beta+1)\right)^2\\
		&&\quad+C^2{\ell}^{2(\gamma-\gamma_1)}\left[\left(\frac{1-\gamma_1}{\gamma-\gamma_1}\right)^{2(1-\gamma_1)}+C^2{\ell}^{2\beta}{(\mathbf{B}(\gamma,\beta))}^2\left(\frac{1-\gamma_1}{\beta+\gamma-\gamma_1}\right)^{2(1-\gamma_1)}\right]\times\\
		&&\quad\|\mathcal{G}_r\|_{{L}^{\frac{1}{2\gamma_1}}[0,\ell]}+C^2Tr(Q)\tilde{L}_h\ell^{2\gamma}\left[\frac{1}{\gamma^2}+C^2{(\mathbf{B}(\gamma,\beta))}^2\frac{\ell^{2\beta}}{(\beta+\gamma)^2}\right](1+E\|\vartheta\|^2_Z)\bigg]\\
		&=& 10m^2_2\Bigg[E\|\vartheta_1\|^2+Tr(Q)\int_{0}^{\ell}E\left\|\kappa(s)\right\|^2ds+C^2(1+\tilde{L}_f)E\big\|\phi(0)\big\|^2\\
		&&\quad+\tilde{L}_f\left(2+E\left\|\vartheta_{\ell}\right\|^2\right)+C^4E\left\|\phi(0)\right\|^2\ell^{2\gamma}\lambda_1^2+C^2{\ell}^{2(\gamma-\gamma_1)}\lambda_2\|\mathcal{G}_r\|_{{L}^{\frac{1}{2\gamma_1}}[0,\ell]}\\
		&&\hspace{2cm}+C^2Tr(Q)\tilde{L}_h\ell^{2\gamma}\lambda_3(1+E\|\vartheta\|^2_Z)\bigg]\\
		&=&l_1+l_2\|\mathcal{G}_r\|_{{L}^{\frac{1}{2\gamma_1}}[0,\ell]}+l_3(1+E\|\vartheta\|^2).
	\end{eqnarray*}
\end{proof}
\begin{theorem} \label{t3.2}
	If the assumptions (C1)-(C5) hold, then the system (\ref{1.1}) is controllable if
	\begin{eqnarray}\label{3.5}
		\dfrac{10C^2\ell^{2(\gamma-\gamma_1)}(m_1^2\lambda_4l_2\ell^{2\gamma_1}+\lambda_2)\delta_1}{1-10\tilde{L}_f-10C^2\ell^{2\gamma}(m_1^2\lambda_4l_3-Tr(Q)\tilde{L}_h\lambda_3)}< 1,
	\end{eqnarray}
	and \begin{eqnarray}\label{3.6}
		\hat{L}_f+C\ell^{\gamma}\lambda_5(2Cm_1m_2\lambda_5\ell^{\gamma}+1)\left(2\hat{L}_g+\hat{L}_h\sqrt{\ell Tr(Q)}\right)<1.
	\end{eqnarray}
\end{theorem}
\begin{proof}Define the operator  $\mathcal{F}:\Theta_r\rightarrow \Theta_r$ given by
	\begin{eqnarray}\label{3.7}
		(\mathcal{F}\vartheta)(t)&=&{A}^{-1}(0)\phi(t), \quad t\in (-\infty,0],\nonumber\\
		(\mathcal{F}\vartheta)(t)&=&{A}^{-1}(0)\phi(0)-f(0,{A}^{-1}(0)\phi(0))+f(t,\vartheta_t)+\displaystyle\int_{0}^{t}\varPsi(t-\varrho,\varrho)\mathcal{U}(\varrho)\phi(0)d\varrho\nonumber\\
		&&\quad+\displaystyle\int_{0}^{t}\varPsi(t-\varrho,\varrho)\big[{B}\upsilon(\varrho)+g\left(\varrho,\vartheta_{\varrho},\vartheta(\varrho)\right)\big]d\varrho +\displaystyle\int_{0}^{t}\varPsi(t-\varrho,\varrho)h(\varrho,\vartheta(\varrho))dw(\varrho)\nonumber\\
		&&\quad+\displaystyle\int_{0}^{t}  \int_{0}^{\varrho}\varPsi(t-\varrho,\varrho)\varPhi(\varrho,s)\big[{B}\upsilon(s)+g\left(s,\vartheta_s,\vartheta(s)\right)\big]dsd\varrho\nonumber\\
		&&\quad+\displaystyle\int_{0}^{t}  \int_{0}^{\varrho}\varPsi(t-\varrho,\varrho)\varPhi(\varrho,s)h(s,\vartheta(s))dw(s)d\varrho,\quad t\in [0,\ell],
	\end{eqnarray} where the control function $\upsilon$ is given by (\ref{3.1}).
	
	\noindent We will show that the operator $\mathcal{F}$ has a fixed point. We prove this result in four steps.
	
	\noindent {\bf Step 1:} In this step,  we show that $\mathcal{F}(\Theta_r)\subset \Theta_r$. Suppose that our claim is not true, then there exist  $t_0 \in J$ and $\tilde{\vartheta}\in\Theta_r$  such that $E\big\|\big(\mathcal{F}\tilde{\vartheta}\big)(t_0)\big\|^2>r$ for every $r>0.$ 
	
	\noindent From (\ref{3.7}), we see that
	\begin{eqnarray}\label{3.8}
		r&<&E\big\|\big(\mathcal{F}\tilde{\vartheta}\big)(t_0)\big\|^2\nonumber\\
		&\leq&10\Bigg[E\left\|{A}^{-1}(0)\phi(0)\right\|^2+E\|f(0,{A}^{-1}(0)\phi(0))\|^2+E\|f(t_0,\tilde{\vartheta}_{t_0})\|^2\nonumber\\
		&&\quad+E\left\|\displaystyle\int_{0}^{t_0}\varPsi(t_0-\varrho,\varrho)\mathcal{U}(\varrho)\phi(0)d\varrho\right\|^2+E\left\|\displaystyle\int_{0}^{t_0}\varPsi(t_0-\varrho,\varrho){B}\upsilon(\varrho)d\varrho\right\|^2\nonumber\\
		&&\quad +E\left\| \displaystyle\int_{0}^{t_0} \varPsi(t_0-\varrho,\varrho)g\left(\varrho,\tilde{\vartheta}_{\varrho},\tilde{\vartheta}(\varrho)\right)d\varrho \right\|^2   +E\left\|\displaystyle\int_{0}^{t_0}\varPsi(t_0-\varrho,\varrho)h(\varrho,\tilde{\vartheta}(\varrho))dw(\varrho)\right\|^2\nonumber\\
		&&\quad+E\left\|\displaystyle\int_{0}^{t_0}  \int_{0}^{\varrho}\varPsi(t_0-\varrho,\varrho)\varPhi(\varrho,s){B}\upsilon(s)dsd\varrho\right\|^2\nonumber\\
		&&\quad+E\left\|\displaystyle\int_{0}^{t_0}  \int_{0}^{\varrho}\varPsi(t_0-\varrho,\varrho)\varPhi(\varrho,s)g\left(s,\tilde{\vartheta}_s,\tilde{\vartheta}(s)\right)dsd\varrho\right\|^2\nonumber\\
		&&\quad+E\left\|\displaystyle\int_{0}^{t_0}  \int_{0}^{\varrho}\varPsi(t_0-\varrho,\varrho)\varPhi(\varrho,s)h(s,\tilde{\vartheta}(s))dw(s)d\varrho\right\|^2\Bigg].
	\end{eqnarray}
	Combining with Remark~\ref{r1}, Lemmas~ \ref{l1},~\ref{l2},~\ref{l3},~\ref{l5} and assumptions (C1)-(C5) in (\ref{3.8}), we get
	\begin{eqnarray}\label{3.9}
		r&\leq&10\Bigg[C^2E\big\|\phi(0)\big\|^2+\tilde{L}_{f}\left(1+C^2E\left\|\phi(0)\right\|^2\right)+\tilde{L}_f\left(1+E\big\|\tilde{\vartheta}_{t_0}\big\|^2\right)\nonumber\\
		&&\quad+C^4E\left\|\phi(0)\right\|^2\left(\int_{0}^{t_0}(t_0-\varrho)^{\gamma-1}(1+\varrho^{\beta})d\varrho\right)^2+C^2m_1^2\left(\int_{0}^{t_0}(t_0-\varrho)^{\gamma-1}E\|\upsilon(\varrho)\|d\varrho\right)^2\nonumber\\	&&\quad+C^2\left(\int_{0}^{t_0}(t_0-\varrho)^{\gamma-1}\left(E\big\|g\big(\varrho,\tilde{\vartheta}_{\varrho},\tilde{\vartheta}(\varrho)\big)\big\|^2\right)^{\frac{1}{2}}d\varrho\right)^{2}\nonumber\\
		&&\quad+C^2Tr(Q)\left(\int_{0}^{t_0}(t_0-\varrho)^{\gamma-1}E\big\|h\big(\varrho,\tilde{\vartheta}(\varrho)\big)\big\|_Qd\varrho\right)^2\nonumber\\
		&&\quad+C^4m_1^2\left(\int_{0}^{t_0}\int_{0}^{\varrho}(t_0-\varrho)^{\gamma-1}(\varrho-s)^{\beta-1}E\|\upsilon(s)\|dsd\varrho \right)^2\nonumber\\
		&&\quad+C^4\left(\int_{0}^{t_0}\int_{0}^{\varrho}(t_0-\varrho)^{\gamma-1}(\varrho-s)^{\beta-1}\left(E\big\|g\big(s,\tilde{\vartheta}_{s},\tilde{\vartheta}(s)\big)\big\|^2\right)^{\frac{1}{2}}dsd\varrho \right)^2\nonumber\\
		&&\quad +C^4Tr(Q)\left(\int_{0}^{t_0}\int_{0}^{\varrho}(t_0-\varrho)^{\gamma-1}(\varrho-s)^{\beta-1}E\|h(s,\tilde{\vartheta}(s))\|_Qdsd\varrho \right)^2
		\Bigg]\nonumber\\
		&\leq&10\Bigg[C^2(1+\tilde{L}_{f})E\big\|\phi(0)\big\|^2+\tilde{L}_f\left(2+E\big\|\tilde{\vartheta}_{t_0}\big\|^2\right)+C^4E\left\|\phi(0)\right\|^2\left(\int_{0}^{t_0}(t_0-\varrho)^{\gamma-1}(1+\varrho^{\beta})d\varrho\right)^2\nonumber\\
		&&\quad+C^2m_1^2\left(\int_{0}^{t_0}(t_0-\varrho)^{\gamma-1}E\|\upsilon(\varrho)\|d\varrho\right)^2+C^2\left(\int_{0}^{t_0}(t_0-\varrho)^{\gamma-1}\left(\mathcal{G}_r(\varrho)\right)^{\frac{1}{2}}d\varrho\right)^{2}\nonumber\\	&&\quad+C^2Tr(Q)\left(\int_{0}^{t_0}(t_0-\varrho)^{\frac{\gamma-1}{2}} \cdot(t_0-\varrho)^{\frac{\gamma-1}{2}}E\big\|h\big(\varrho,\tilde{\vartheta}(\varrho)\big)\big\|_Qd\varrho\right)^2\nonumber\\
		&&\quad+C^4m_1^2{(\mathbf{B}(\gamma,\beta))}^2\left(\int_{0}^{t_0}(t_0-\varrho)^{\beta+\gamma-1}E\|\upsilon(\varrho)\|d\varrho \right)^2\nonumber\\
		&&\quad+C^4{(\mathbf{B}(\gamma,\beta))}^2\left(\int_{0}^{t_0}(t_0-\varrho)^{\beta+\gamma-1}\left(\mathcal{G}_r(\varrho)\right)^{\frac{1}{2}}d\varrho \right)^2\nonumber\\
		&&\quad+C^4Tr(Q){(\mathbf{B}(\gamma,\beta))}^2\left(\int_{0}^{t_0}(t_0-\varrho)^{\frac{\beta+\gamma-1}{2}} \cdot(t_0-\varrho)^{\frac{\beta+\gamma-1}{2}} E\big\|h(\varrho,\tilde{\vartheta}(\varrho))\big\|_Qd\varrho \right)^2\Bigg].
	\end{eqnarray}
	Using H\"{o}lder inequality, Cauchy-Schwarz inequality and Lemma~\ref{l4}~ in the equation (\ref{3.9}), we get
	\begin{eqnarray*}
		r&\leq&10\Bigg[C^2(1+\tilde{L}_{f})E\big\|\phi(0)\big\|^2+\tilde{L}_f\left(2+E\big\|\tilde{\vartheta}_{t_0}\big\|^2\right)+C^4E\left\|\phi(0)\right\|^2t_0^{2\gamma} \left(\frac{1}{\gamma}+{t_0}^{\beta}\mathbf{B}(\gamma,\beta+1)\right)^2\\
		&&\quad+m_1^2C^2\left(\int_{0}^{t_0}(t_0-\varrho)^{2(\gamma-1)}d\varrho\right)\left(\int_{0}^{t_0}E\|\upsilon(\varrho)\|^2d\varrho\right) \\
		&&\quad+C^2\left( \int_{0}^{t_0} (t_0-\varrho)^{\frac{\gamma-1}{1-\gamma_1}}d\varrho \right)^{2(1-\gamma_1)} \left(\int_{0}^{t_0}\mathcal{G}^{\frac{1}{2\gamma_1}}_r(\varrho)d\varrho\right)^{2\gamma_1} \\ &&\quad+C^2Tr(Q)\left(\int_{0}^{t_0}(t_0-\varrho)^{\gamma-1}d\varrho\right)\left(\int_{0}^{t_0}(t_0-\varrho)^{\gamma-1}E\big\|h\big(\varrho,\tilde{\vartheta}(\varrho)\big)\big\|_Q^2d\varrho\right)\\
		&&\quad+C^4m_1^2{(\mathbf{B}(\gamma,\beta))}^2\left( \int_{0}^{t_0} (t_0-\varrho)^{2(\beta+\gamma-1)}d\varrho \right) \left(\int_{0}^{t_0}E\|\upsilon(\varrho)\|^2d\varrho\right)\\
		&&\quad+C^4{(\mathbf{B}(\gamma,\beta))}^2\left( \int_{0}^{t_0} (t_0-\varrho)^{\frac{\beta+\gamma-1}{1-\gamma_1}}d\varrho \right)^{2(1-\gamma_1)} \left(\int_{0}^{t_0}\mathcal{G}^{\frac{1}{2\gamma_1}}_r(\varrho)d\varrho\right)^{2\gamma_1}\\
		&&\quad+C^4Tr(Q){(\mathbf{B}(\gamma,\beta))}^2\left(\int_{0}^{t_0}(t_0-\varrho)^{\beta+\gamma-1}d\varrho\right)\left(\int_{0}^{t_0}(t_0-\varrho)^{\beta+\gamma-1}E\big\|h\big(\varrho,\tilde{\vartheta}(\varrho)\big)\big\|_Q^2d\varrho\right)\Bigg] \\
		&\leq&10\Bigg[C^2(1+\tilde{L}_{f})E\big\|\phi(0)\big\|^2+\tilde{L}_f\left(2+E\big\|\tilde{\vartheta}_{\ell}\big\|^2\right)+C^4E\left\|\phi(0)\right\|^2{\ell}^{2\gamma}\left(\frac{1}{\gamma}+{\ell}^{\beta}\mathbf{B}(\gamma,\beta+1)\right)^2\\
		&&\quad+\frac{m_1^2C^2}{2\gamma-1}\left(l_1+l_2\|\mathcal{G}_r\|_{{L}^{\frac{1}{2\gamma_1}}[0,\ell]}+l_3\left(1+E\big\|\tilde{\vartheta}\big\|^2\right)\right)\ell^{2\gamma}\\
		&&\quad+C^2\left(\frac{1-\gamma_1}{\gamma-\gamma_1}\right)^{2(1-\gamma_1)}{\ell}^{2(\gamma-\gamma_1)}\|\mathcal{G}_r\|_{{L}^{\frac{1}{2\gamma_1}}[0,\ell]} +C^2Tr(Q)\tilde{L}_h\frac{\ell^{2\gamma}}{\gamma^2}\left(1+E\big\|\tilde{\vartheta}\big\|^2_Z\right)\\
		&&\quad+\frac{C^4m_1^2{(\mathbf{B}(\gamma,\beta))}^2}{2\gamma+2\beta-1}\left(l_1+l_2\|\mathcal{G}_r\|_{{L}^{\frac{1}{2\gamma_1}}[0,\ell]}+l_3\left(1+E\big\|\tilde{\vartheta}\big\|^2\right)\right)\ell^{2(\gamma+\beta)}\\			&&\quad+C^4{(\mathbf{B}(\gamma,\beta))}^2\left(\frac{1-\gamma_1}{\beta+\gamma-\gamma_1}\right)^{2(1-\gamma_1)}{\ell}^{2(\beta+\gamma-\gamma_1)}\|\mathcal{G}_r\|_{{L}^{\frac{1}{2\gamma_1}}[0,\ell]}\\
		&&\quad+C^4Tr(Q){(\mathbf{B}(\gamma,\beta))}^2\tilde{L}_h\frac{\ell^{2(\beta+\gamma)}}{(\beta+\gamma)^2}\left(1+E\|\tilde{\vartheta}\|^2_Z\right)\Bigg]\\
		&\leq&10\Bigg[C^2\left((1+\tilde{L}_{f})+C^2{\ell}^{2\gamma}\lambda_1^2\right)E\left\|\phi(0)\right\|^2+\tilde{L}_f(2+r)\\
		&&\quad+m_1^2C^2\ell^{2\gamma}\left[\frac{1}{2\gamma-1}+\ell^{2\beta}\frac{C^2{(\mathbf{B}(\gamma,\beta))}^2}{2\gamma+2\beta-1}\right]\left(l_1+l_2\|\mathcal{G}_r\|_{{L}^{\frac{1}{2\gamma_1}}[0,\ell]}+l_3\left(1+r\right)\right)\\
		&&\quad+C^2{\ell}^{2(\gamma-\gamma_1)}\left[\left(\frac{1-\gamma_1}{\gamma-\gamma_1}\right)^{2(1-\gamma_1)}+C^2{\ell}^{2\beta}{(\mathbf{B}(\gamma,\beta))}^2\left(\frac{1-\gamma_1}{\beta+\gamma-\gamma_1}\right)^{2(1-\gamma_1)}\right]\|\mathcal{G}_r\|_{{L}^{\frac{1}{2\gamma_1}}[0,\ell]}\\
		&&\quad+C^2Tr(Q)\tilde{L}_h\ell^{2\gamma}\left[\frac{1}{\gamma^2}+C^2{(\mathbf{B}(\gamma,\beta))}^2\frac{\ell^{2\beta}}{(\beta+\gamma)^2}\right](1+r)\Bigg]\\	
		&=&10\Bigg[C^2\left((1+\tilde{L}_{f})+C^2{\ell}^{2\gamma}\lambda_1^2\right)E\left\|\phi(0)\right\|^2+\tilde{L}_f(2+r)\\
		&&\quad+m_1^2C^2\ell^{2\gamma}\lambda_4\left(l_1+l_2\|\mathcal{G}_r\|_{{L}^{\frac{1}{2\gamma_1}}[0,\ell]}+l_3\left(1+r\right)\right)\\
		&&\quad+C^2{\ell}^{2(\gamma-\gamma_1)}\lambda_2\|\mathcal{G}_r\|_{{L}^{\frac{1}{2\gamma_1}}[0,\ell]}+C^2Tr(Q)\tilde{L}_h\ell^{2\gamma}\lambda_3(1+r)\Bigg].
	\end{eqnarray*}
	Dividing by $r$ and taking $\liminf$ as $r\rightarrow \infty,$ we get 
	\begin{eqnarray*}
		1&\leq& 10\bigg[\tilde{L}_f+m_1^2C^2\ell^{2\gamma}\lambda_4\left(l_2\delta_1+l_3\right)+C^2{\ell}^{2(\gamma-\gamma_1)}\lambda_2\delta_1+C^2Tr(Q)\tilde{L}_h\ell^{2\gamma}\lambda_3\bigg],
	\end{eqnarray*}
	or 	\begin{eqnarray*}
		1&\leq&\dfrac{10C^2\ell^{2(\gamma-\gamma_1)}(m_1^2\lambda_4l_2\ell^{2\gamma_1}+\lambda_2)\delta_1}{1-10\tilde{L}_f-10C^2\ell^{2\gamma}(m_1^2\lambda_4l_3-Tr(Q)\tilde{L}_h\lambda_3)} ,
	\end{eqnarray*}
	which contradicts to our assumption.  Therefore, $E\big\|\big(\mathcal{F}\tilde{\vartheta}\big)(t_0)\big\|^2<r$ for $t_0 \in J,~r>0.$ It is also true for $t_0 \in (-\infty,0].$  Hence $\mathcal{F}(\Theta_r)\subset \Theta_r$. 
	
	\noindent{\bf Step 2:} $\mathcal{F}$ is continuous on $\Theta_r.$ For $t\in (-\infty,0]$, it is trivial. We consider a sequence $\{\vartheta^n\}_{n\in \mathbb{N}}$ in $\Theta_r$ such that $\lim_{n\rightarrow\infty}E\|\vartheta^n-\vartheta\|^2=0$ in $\Theta_r.$  By the continuity of the nonlinear functions $f,g$ and $h$, we get that	
	\begin{itemize}
		\item[(i)]	$\lim\limits_{n\rightarrow\infty}E\left\|f\big(t,\vartheta_t^n\big)-f\big(t,\vartheta_t\big)\right\|^2=0,$
		\item[(ii)]$\lim\limits_{n\rightarrow\infty}E\left\|g\big(t,\vartheta_t^n,\vartheta^n(t)\big)-g\big(t,\vartheta_t,\vartheta(t)\big)\right\|^2=0$,
		\item[(iii)]$ \lim\limits_{n\rightarrow\infty}E\left\|h\big(t,\vartheta^n(t)\big)-h\big(t,\vartheta(t)\big)\right\|_Q^2=0.$
	\end{itemize}
	From the equation (\ref{3.7}) and Lemma~\ref{l3}, we get\\
	
	$E\|(\mathcal{F}\vartheta^n)(t)-(\mathcal{F}\vartheta)(t) \|^2$ 
	\begin{eqnarray}
		&\leq&5\Bigg[E\left\|f\big(t,\vartheta^n_t\big)-f\big(t,\vartheta_t\big)\right\|^2+E\left\|\int_{0}^{t} \varPsi(t-\varrho,\varrho)\Big[g\big(\varrho,\vartheta^n_{\varrho},\vartheta^n(\varrho)\big)-g\big(\varrho,\vartheta_{\varrho},\vartheta(\varrho)\big)\Big]d\varrho\right\|^2 \nonumber\\
		&&\quad+E\left\|\int_{0}^{t} \varPsi(t-\varrho,\varrho)\Big[h\big(\varrho,\vartheta^n(\varrho)\big)-h\big(\varrho,\vartheta(\varrho)\big)\Big]d\varrho\right\|^2 \nonumber\\
		&&\quad +E\left\|\displaystyle\int_{0}^{t}  \int_{0}^{\varrho}\varPsi(t-\varrho,\varrho)\varPhi(\varrho,s)\Big[g\big(s,\vartheta^n_{s},\vartheta^n(s)\big)-g\big(s,\vartheta_{s},\vartheta(s)\big)\Big]dsd\varrho\right\|^2 \nonumber \\ 
		&&\quad+E\left\|\displaystyle\int_{0}^{t}  \int_{0}^{\varrho}\varPsi(t-\varrho,\varrho)\varPhi(\varrho,s)\Big[h\big(s,\vartheta^n(s)\big)-h\big(s,\vartheta(s)\big)\Big]dw(s)d\varrho\right\|^2\Bigg]\nonumber\\
		&\leq&5\Bigg[L_fE\big\|\vartheta^n_t-\vartheta_t\big\|^2+C^2\left(\int_{0}^{t}(t-\varrho)^{\gamma-1}E\left\|g\big(\varrho,\vartheta^n_{\varrho},\vartheta^n(\varrho)\big)-g\big(\varrho,\vartheta_{\varrho},\vartheta(\varrho)\big) \right\|d\varrho\right)^2 \nonumber\\
		&&\quad+C^2Tr(Q)\left(\int_{0}^{t}(t-\varrho)^{\gamma-1}E\left \|h\big(\varrho,\vartheta^n(\varrho)\big)-h\big(\varrho,\vartheta(\varrho)\big)\right\|_Qd\varrho\right)^2 \nonumber\\
		&&\quad+C^4\left(\int_{0}^{t}\int_{0}^{\varrho}(t-\varrho)^{\gamma-1}(\varrho-s)^{\beta-1}E\left\|g\big(s,\vartheta^n_{s},\vartheta^n(s)\big)-g\big(s,\vartheta_{s},\vartheta(s)\big)\right\|dsd\varrho \right)^2\nonumber\\
		&&\quad+C^4Tr(Q)\left(\int_{0}^{t}\int_{0}^{\varrho}(t-\varrho)^{\gamma-1}(\varrho-s)^{\beta-1}E\left\|h\big(\varrho,\vartheta^n(\varrho)\big)-h\big(\varrho,\vartheta(\varrho)\big)\right\|_Qdsd\varrho \right)^2\Bigg]\nonumber\\
		&\leq& 5\Bigg[L_fE\big\|\vartheta^n_t-\vartheta_t\big\|^2+C^2\left(\int_{0}^{t}(t-\varrho)^{\gamma-1}E\left\|g\big(\varrho,\vartheta^n_{\varrho},\vartheta^n(\varrho)\big)-g\big(\varrho,\vartheta_{\varrho},\vartheta(\varrho)\big) \right\|d\varrho\right)^2 \nonumber\\
		&&\quad+C^2Tr(Q)\left(\int_{0}^{t}(t-\varrho)^{\gamma-1}E\left \|h\big(\varrho,\vartheta^n(\varrho)\big)-h\big(\varrho,\vartheta(\varrho)\big)\right\|_Qd\varrho\right)^2 \nonumber\\
		&&\quad+C^4\left(\int_{0}^{t}(t-\varrho)^{\beta+\gamma-1}E\left\|g\big(\varrho,\vartheta^n_{\varrho},\vartheta^n(\varrho)\big)-g\big(\varrho,\vartheta_{\varrho},\vartheta(\varrho)\big)\right\|d\varrho \right)^2\nonumber\\
		&&\quad+C^4Tr(Q)\left(\int_{0}^{t}(t-\varrho)^{\beta+\gamma-1}E\|h\big(\varrho,\vartheta^n(\varrho)\big)-h\big(\varrho,\vartheta(\varrho)\big)\|_Qd\varrho \right)^2\Bigg]\nonumber\\
		&\leq&5\Bigg[L_fE\big\|\vartheta^n_t-\vartheta_t\big\|^2+\left(\int_{0}^{t}(t-\varrho)^{2(\gamma-1)}d\varrho\right)\left(\int_{0}^{t}E\left\|g\big(\varrho,\vartheta^n_{\varrho},\vartheta^n(\varrho)\big)-g\big(\varrho,\vartheta_{\varrho},\vartheta(\varrho)\big) \right\|^2d\varrho\right)\nonumber\\
		&&\quad+C^2Tr(Q)\left(\int_{0}^{t}(t-\varrho)^{2(\gamma-1)}d\varrho\right)\left(\int_{0}^{t}E\left\|h\big(\varrho,\vartheta^n(\varrho)\big)-h\big(\varrho,\vartheta(\varrho)\big) \right\|^2_Qd\varrho\right) \nonumber\\
		&&\quad+C^4\left(\int_{0}^{t}(t-\varrho)^{2(\beta+\gamma-1)}d\varrho\right)\left(\int_{0}^{t}E\left\|g\big(\varrho,\vartheta^n_{\varrho},\vartheta^n(\varrho)\big)-g\big(\varrho,\vartheta_{\varrho},\vartheta(\varrho)\big) \right\|^2d\varrho\right)\nonumber\\
		&&\quad+C^4Tr(Q)\left(\int_{0}^{t}(t-\varrho)^{2(\beta+\gamma-1)}d\varrho\right)\left(\int_{0}^{t}E\left\|h\big(\varrho,\vartheta^n(\varrho)\big)-h\big(\varrho,\vartheta(\varrho)\big) \right\|^2_Qd\varrho\right)\Bigg].\nonumber
	\end{eqnarray}
	Taking limit as $n \rightarrow \infty,$ we get
	$$E\| \mathcal{F}\vartheta^n-\mathcal{F}\vartheta\|^2 \rightarrow 0 \quad \mbox{as} \quad n \rightarrow \infty,$$
	which implies that $\mathcal{F}$ is continuous on $\Theta_r$ for $t\in J.$ 
	
	\noindent{\bf Step 3:} $\mathcal{F}(\Theta_r)$ is equicontinuous. It is trivial to show that  $\mathcal{F}(\Theta_r)$ is equicontinuous for $t\in (-\infty,0]$. Next, we show that  $\mathcal{F}(\Theta_r)$ is equicontinuous on $J.$ So, for every $\vartheta \in \mathcal{F}(\Theta_r)$ and $t_1, t_2 \in J$ such that $0\leq t_1< t_2\leq\ell,$ we have\\ 
	
	$E\|(\mathcal{F}\vartheta)(t_2)-(\mathcal{F}\vartheta)(t_1) \|^2$
	\begin{eqnarray*} 
		&\leq&15\Bigg[E\|f(t_2,\vartheta_{t_2})-f(t_1,\vartheta_{t_1})\|^2+E\left\|\int_{t_1}^{t_2}\varPsi(t_2-\varrho,\varrho) \mathcal{U}(\varrho)\phi(0)d\varrho\right\|^2\\
		&&\quad+E\left\|\int_{0}^{t_1}\Big[\varPsi(t_2-\varrho,\varrho)-\varPsi(t_1-\varrho,\varrho)\Big]\mathcal{U}(\varrho)\phi(0)d\varrho\right\|^2\\
		&&\quad+E\bigg\|\int_{t_1}^{t_2}\varPsi(t_2-\varrho,\varrho){B}\upsilon(\varrho)d\varrho\bigg\|^2+E\bigg\|\int_{0}^{t_1}\Big[\varPsi(t_2-\varrho,\varrho)-\varPsi(t_1-\varrho,\varrho)\Big]{B}\upsilon(\varrho)d\varrho\bigg\|^2\\	
		&&\quad+E\bigg\| \int_{t_1}^{t_2} \varPsi(t_2-\varrho,\varrho)g\left(\varrho,\vartheta_{\varrho},\vartheta(\varrho)\right)d\varrho\bigg\|^2\\
		&&\quad+E\left\| \int_{0}^{t_1}\Big[\varPsi(t_2-\varrho,\varrho)-\varPsi(t_1-\varrho,\varrho)\Big] g\left(\varrho,\vartheta_{\varrho},\vartheta(\varrho)\right)d\varrho\right\|^2\\
		&&\quad+E\bigg\| \int_{t_1}^{t_2} \varPsi(t_2-\varrho,\varrho)h\left(\varrho,\vartheta(\varrho)\right)dw(\varrho)\bigg\|^2\\
		&&\quad+E\bigg\| \int_{0}^{t_1}\Big[ \varPsi(t_2-\varrho,\varrho)-\varPsi(t_1-\varrho,\varrho)\Big]h\left(\varrho,\vartheta(\varrho)\right)dw(\varrho)\bigg\|^2\\
		&&\quad+E\bigg\|\int_{t_1}^{t_2}  \int_{0}^{\varrho}\varPsi(t_2-\varrho,\varrho)\varPhi(\varrho,s){B}\upsilon(s)dsd\varrho\bigg\|^2\\
		&&\quad+E\bigg\|\int_{0}^{t_1}  \int_{0}^{\varrho}\Big[\varPsi(t_2-\varrho,\varrho)-\varPsi(t_1-\varrho,\varrho)\Big]\varPhi(\varrho,s){B}\upsilon(s)dsd\varrho\bigg\|^2\\
		&&\quad+E\bigg\|\int_{t_1}^{t_2}  \int_{0}^{\varrho}\varPsi(t_2-\varrho,\varrho)\varPhi(\varrho,s)g\left(s,\vartheta_s,\vartheta(s)\right)dsd\varrho\bigg\|^2 \\ 
		&&\quad+E\bigg\|\int_{0}^{t_1}  \int_{0}^{\varrho}\Big[\varPsi(t_2-\varrho,\varrho)-\varPsi(t_1-\varrho,\varrho)\Big]\varPhi(\varrho,s)g\left(s,\vartheta_s,\vartheta(s)\right)dsd\varrho\bigg\|^2 \\ 
		&&\quad+E\bigg\|\int_{t_1}^{t_2}  \int_{0}^{\varrho}\varPsi(t_2-\varrho,\varrho)\varPhi(\varrho,s)h\left(s,\vartheta(s)\right)dw(s)d\varrho\bigg\|^2\\
		&&\quad+E\bigg\|\int_{0}^{t_1}  \int_{0}^{\varrho}\Big[\varPsi(t_2-\varrho,\varrho)-\varPsi(t_1-\varrho,\varrho)\Big]\varPhi(\varrho,s)h\left(s,\vartheta(s)\right)dw(s)d\varrho\bigg\|^2\Bigg] \\ 
		&\leq&\sum\limits_{j=1}^{15}I_j,   
	\end{eqnarray*}
	where
	\begin{eqnarray*} 
		I_1&=&E\|f(t_2,\vartheta_{t_2})-f(t_1,\vartheta_{t_1})\|^2,~I_2=\left(\int_{t_1}^{t_2}E\left\|\varPsi(t_2-\varrho,\varrho) \mathcal{U}(\varrho)\phi(0)\right\|d\varrho\right)^2,\\
		I_3&=&\left(\int_{0}^{t_1}E\left\|\Big[\varPsi(t_2-\varrho,\varrho)-\varPsi(t_1-\varrho,\varrho)\Big]\mathcal{U}(\varrho)\phi(0)\right\|d\varrho\right)^2,\\
		I_4&=&\left(\int_{t_1}^{t_2}E\left\|\varPsi(t_2-\varrho,\varrho){B}\upsilon(\varrho)\right\|d\varrho\right)^2,\\
		I_5&=&\left(\int_{0}^{t_1}E\left\|\Big[\varPsi(t_2-\varrho,\varrho)-\varPsi(t_1-\varrho,\varrho)\Big]{B}\upsilon(\varrho)\right\|d\varrho\right)^2,\\	
		I_6&=&\left(\int_{t_1}^{t_2}E\left\|  \varPsi(t_2-\varrho,\varrho)g\left(\varrho,\vartheta_{\varrho},\vartheta(\varrho)\right)\right\|d\varrho\right)^2,\\
		I_7&=&\left( \int_{0}^{t_1}E\left\|\Big[\varPsi(t_2-\varrho,\varrho)-\varPsi(t_1-\varrho,\varrho)\Big] g\left(\varrho,\vartheta_{\varrho},\vartheta(\varrho)\right)\right\|d\varrho\right)^2,\\
		I_8&=&Tr(Q)\left(\int_{t_1}^{t_2}E\left\| \varPsi(t_2-\varrho,\varrho)h\left(\varrho,\vartheta(\varrho)\right)\right\|d\varrho\right)^2,\\
		I_9&=&Tr(Q)\left(\int_{0}^{t_1}E\left\|\Big[ \varPsi(t_2-\varrho,\varrho)-\varPsi(t_1-\varrho,\varrho)\Big]h\left(\varrho,\vartheta(\varrho)\right)\right\|d\varrho\right)^2,\\
		I_{10}&=&\left(\int_{t_1}^{t_2}  \int_{0}^{\varrho}E\left\|\varPsi(t_2-\varrho,\varrho)\varPhi(\varrho,s){B}\upsilon(s)\right\|dsd\varrho\right)^2,\\
		I_{11}&=&\left(\int_{0}^{t_1}  \int_{0}^{\varrho}E\left\|\Big[\varPsi(t_2-\varrho,\varrho)-\varPsi(t_1-\varrho,\varrho)\Big]\varPhi(\varrho,s){B}\upsilon(s)\right\|dsd\varrho\right)^2,\\
		I_{12}&=&\left(\int_{t_1}^{t_2}  \int_{0}^{\varrho}E\left\|\varPsi(t_2-\varrho,\varrho)\varPhi(\varrho,s)g\left(s,\vartheta_s,\vartheta(s)\right)\right\|dsd\varrho\right)^2, \\ 
		I_{13}&=&\left(\int_{0}^{t_1}  \int_{0}^{\varrho}E\left\|\Big[\varPsi(t_2-\varrho,\varrho)-\varPsi(t_1-\varrho,\varrho)\Big]\varPhi(\varrho,s)g\left(s,\vartheta_s,\vartheta(s)\right)\right\|dsd\varrho\right)^2, \\ 
		I_{14}&=&Tr(Q)\left(\int_{t_1}^{t_2}  \int_{0}^{\varrho}E\left\|\varPsi(t_2-\varrho,\varrho)\varPhi(\varrho,s)h\left(s,\vartheta(s)\right)\right\|dsd\varrho\right)^2,\\
		I_{15}&=&Tr(Q)\left(\int_{0}^{t_1}  \int_{0}^{\varrho}E\left\|\Big[\varPsi(t_2-\varrho,\varrho)-\varPsi(t_1-\varrho,\varrho)\Big]\varPhi(\varrho,s)h\left(s,\vartheta(s)\right)\right\|dsd\varrho\right)^2.  
	\end{eqnarray*}
	Therefore, we only need to prove that $I_j\rightarrow0$ independent of $\vartheta\in \Theta_r$ as $t_1-t_2\rightarrow0$ for $j=1,2,\dots,15.$ 
	
	\noindent For $I_1$, using the assumption (C3), we have 
	\begin{eqnarray*}
		I_1&\leq&L_1(|t_2-t_1|+E\|\vartheta_{t_2}-\vartheta_{t_1}\|^2_{\wp})~\rightarrow~ 0~ \mbox{as }~ t_1\rightarrow t_2.
	\end{eqnarray*}
	For $I_2$, using the Lemma~\ref{l3}, we have 
	\begin{eqnarray*}
		I_2&\leq&C^4E\left\|\phi(0)\right\|^2\left(\int_{t_1}^{t_2}(t_2-\varrho)^{\gamma-1}(1+\varrho^{\beta})d\varrho\right)^2~\rightarrow~0~ \mbox{as }~ t_1\rightarrow t_2.
	\end{eqnarray*}
	For $t_1=0$ and $0 <t_2\leq\ell$, we can easily show that $I_3=0$. For $t_1>0$ and $\epsilon>0$, by Lemma \ref{l3} and using the continuity of the operator $\varPsi(t-\varrho,\varrho)$ in the uniform topology with respect to the variables $t$ and $\varrho$, where $0 \leq t\leq\ell$ and $0\leq \varrho\leq t-\epsilon$, we have 
	\begin{eqnarray*}
		I_3&\leq&\sup\limits_{\varrho\in [0,t_1-\epsilon]}\left\|\varPsi(t_2-\varrho,\varrho)-\varPsi(t_1-\varrho,\varrho)\right\|^2C^2E\|\phi(0)\|^2\left(\int_{0}^{t_1-\epsilon}(1+\varrho^{\beta})d\varrho \right)^2\\
		&&\quad+C^4E\|\phi(0)\|^2\left(\int_{t_1-\epsilon}^{t_1}\left[(t_2-\varrho)^{\gamma-1}+(t_1-\varrho)^{\gamma-1}\right](1+\varrho^{\beta})d\varrho\right)^2\\
		&\rightarrow& 0~ \mbox{as }~ t_1\rightarrow t_2 \mbox{ and } \epsilon \rightarrow 0.
	\end{eqnarray*}
	For $I_4$, using the assumption (C1), Lemma \ref{l3}, and H\"{o}lder inequality, we have
	\begin{eqnarray*}
		I_4&\leq&C^2m_1^2\left(\int_{t_1}^{t_2}(t_2-\varrho)^{\gamma-1}E\|\upsilon(\varrho)\|d\varrho\right)^2\nonumber\\&\leq&m_1^2C^2\left(\int_{t_1}^{t_2}(t_2-\varrho)^{2(\gamma-1)}d\varrho\right)\left(\int_{t_1}^{t_2}E\|\upsilon(\varrho)\|^2d\varrho\right)\\
		&\leq&\frac{m_1^2C^2}{2\gamma-1}\left(l_1+l_2\|\mathcal{G}_r\|_{{L}^{\frac{1}{2\gamma_1}}[0,\ell]}+l_3\left(1+E\|\vartheta\|^2\right)\right)(t_2-t_1)^{2\gamma}\\
		&\rightarrow& 0~ \mbox{as }~ t_1-t_2\rightarrow0.
	\end{eqnarray*}
	For $t_1=0$ and $0 <t_2\leq\ell$, we can easily show that $I_5=0$. For $t_1>0$ and $\epsilon>0$, by Lemma \ref{l3}, assumption (C1), and using the continuity of the operator $\varPsi(t-\varrho,\varrho)$ in the uniform topology with respect to the variables $t$ and $\varrho$, where $0 \leq t\leq\ell$ and $0\leq \varrho\leq t-\epsilon$, we have 
	\begin{eqnarray*}
		I_5&\leq&\sup\limits_{\varrho\in [0,t_1-\epsilon]}\left\|\varPsi(t_2-\varrho,\varrho)-\varPsi(t_1-\varrho,\varrho)\right\|^2 m_1^2 \left(\int_{0}^{t_1-\epsilon}E\|\upsilon(\varrho)\|d\varrho \right)^2\\
		&&\quad+C^2\left(\int_{t_1-\epsilon}^{t_1}\left[(t_2-\varrho)^{\gamma-1}+(t_1-\varrho)^{\gamma-1}\right]E\|\upsilon(\varrho)\|d\varrho\right)^2\\
		&\rightarrow& 0~ \mbox{as }~ t_1\rightarrow t_2 \mbox{ and } \epsilon \rightarrow 0.
	\end{eqnarray*}
	For $I_6$, using the assumption (C4), Lemma~\ref{l3}, and H\"{o}lder inequality, we find that
	\begin{eqnarray*}
		I_6&\leq&C^2\left(\int_{t_1}^{t_2}(t_2-\varrho)^{\gamma-1}\left(E\left\|g\left(\varrho,\vartheta_{\varrho},\vartheta(\varrho)\right)\right\|^2\right)^{\frac{1}{2}}d\varrho\right)^{2}\nonumber\\
		&\leq&C^2\left(\int_{t_1}^{t_2}(t_2-\varrho)^{\gamma-1}\left(\mathcal{G}_r(\varrho)\right)^{\frac{1}{2}}d\varrho\right)^{2}\nonumber\\
		&\leq&C^2\left( \int_{t_1}^{t_2} (t_2-\varrho)^{\frac{\gamma-1}{1-\gamma_1}}d\varrho \right)^{2(1-\gamma_1)} \left(\int_{t_1}^{t_2}\mathcal{G}^{\frac{1}{2\gamma_1}}_r(\varrho)d\varrho\right)^{2\gamma_1} \\
		&\leq&C^2\left(\frac{1-\gamma_1}{\gamma-\gamma_1}\right)^{2(1-\gamma_1)}{(t_2-t_1)}^{2(\gamma-\gamma_1)}\|\mathcal{G}_r\|_{{L}^{\frac{1}{2\gamma_1}}[0,\ell]}\\
		&\rightarrow& 0~ \mbox{as }~ t_2-t_1\rightarrow0.
	\end{eqnarray*}
	For $t_1=0$ and $0 <t_2\leq\ell$, we can easily show that $I_7=0$. For $t_1>0$ and $\epsilon>0$, by Lemma \ref{l3}, assumption (C4), and the continuity of the operator $\varPsi(t-\varrho,\varrho)$ in uniform topology with respect to the variables $t$ and $\varrho$, where $0 \leq t\leq\ell$ and $0\leq \varrho\leq t-\epsilon$, we get
	\begin{eqnarray*}
		I_7&\leq&\sup\limits_{\varrho\in [0,t_1-\epsilon]}\left\|\varPsi(t_2-\varrho,\varrho)-\varPsi(t_1-\varrho,\varrho)\right\|^2 \left(\int_{0}^{t_1-\epsilon}\left(E\left\|g\left(\varrho,\vartheta_{\varrho},\vartheta(\varrho)\right)\right\|^2\right)^{\frac{1}{2}}d\varrho \right)^2\\
		&&\quad+C^2\left(\int_{t_1-\epsilon}^{t_1}\left[(t_2-\varrho)^{\gamma-1}+(t_1-\varrho)^{\gamma-1}\right]\left(E\left\|g\left(\varrho,\vartheta_{\varrho},\vartheta(\varrho)\right)\right\|^2\right)^{\frac{1}{2}}d\varrho\right)^2\\
		&\rightarrow& 0~ \mbox{as }~ t_1\rightarrow t_2 \mbox{ and } \epsilon \rightarrow 0.
	\end{eqnarray*}
	For $I_8$, by using the assumption (C5), Lemma \ref{l3}, and Cauchy-Schwarz inequality, we find that
	\begin{eqnarray*}
		I_8&\leq&C^2Tr(Q)\left(\int_{t_1}^{t_2}(t_2-\varrho)^{\gamma-1}E\left\|h\left(\varrho,\vartheta(\varrho)\right)\right\|_Qd\varrho\right)^{2}\nonumber\\
		&\leq&C^2Tr(Q)\left(\int_{t_1}^{t_2}(t_2-\varrho)^{\frac{\gamma-1}{2}} \cdot(t_2-\varrho)^{\frac{\gamma-1}{2}}E\left\|h\left(\varrho,\vartheta(\varrho)\right)\right\|_Qd\varrho\right)^2\nonumber\\
		&\leq&C^2Tr(Q)\left(\int_{t_1}^{t_2}(t_2-\varrho)^{\gamma-1}d\varrho\right)\left(\int_{t_1}^{t_2}(t_2-\varrho)^{\gamma-1}E\left\|h\left(\varrho,\vartheta(\varrho)\right)\right\|_Q^2d\varrho\right)\\
		&\leq&C^2Tr(Q)\tilde{L}_h\frac{(t_2-t_1)^{2\gamma}}{\gamma^2}(1+E\|\vartheta\|^2_Z)\\
		&\rightarrow& 0~ \mbox{as }~ t_2-t_1\rightarrow0.
	\end{eqnarray*}
	For $t_1=0$ and $0 <t_2\leq\ell$, we can easily show that $I_9=0$. For $t_1>0$ and $\epsilon>0$, by Lemma \ref{l3}, assumption (C5), and using the continuity of operator $\varPsi(t-\varrho,\varrho)$ in uniform topology with respect to the variables $t$ and $\varrho$, where $0 \leq t\leq\ell$ and $0\leq \varrho\leq t-\epsilon$, we have 
	\begin{eqnarray*}
		I_9&\leq&\sup\limits_{\varrho\in [0,t_1-\epsilon]}\left\|\varPsi(t_2-\varrho,\varrho)-\varPsi(t_1-\varrho,\varrho)\right\|^2Tr(Q) \left(\int_{0}^{t_1-\epsilon}E\left\|h\left(\varrho,\vartheta(\varrho)\right)\right\|_Qd\varrho \right)^2\\
		&&\quad+C^2Tr(Q)\left(\int_{t_1-\epsilon}^{t_1}\left[(t_2-\varrho)^{\gamma-1}+(t_1-\varrho)^{\gamma-1}\right]E\left\|h\left(\varrho,\vartheta(\varrho)\right)\right\|_Qd\varrho\right)^2\\
		&\rightarrow& 0~ \mbox{as }~ t_2\rightarrow t_1 \mbox{ and } \epsilon \rightarrow 0.
	\end{eqnarray*}
	For $I_{10}$, by using the assumption (C1), Lemma~\ref{l3}, and since the function $\varrho\rightarrow (t_2-\varrho)^{\gamma-1}I^{\beta}_{\varrho}E\|\upsilon(\varrho)\|$ is Lebesgue integrable, we have
	\begin{eqnarray*}
		I_{10}&\leq&C^4m_1^2\left(\int_{t_1}^{t_2}\int_{0}^{\varrho}(t_2-\varrho)^{\gamma-1}(\varrho-s)^{\beta-1}E\|\upsilon(s)\|dsd\varrho \right)^2\nonumber\\
		&\leq&C^4m_1^2(\Gamma(\beta))^2\left(\int_{t_1}^{t_2}(t_2-\varrho)^{\gamma-1}I^{\beta}_{\varrho}E\|\upsilon(\varrho)\|d\varrho \right)^2\nonumber\\
		&\rightarrow& 0~ \mbox{as }~ t_2-t_1\rightarrow0.
	\end{eqnarray*}
	For $t_1=0$ and $0 <t_2\leq\ell$, we can easily show that $I_{11}=0$. For $t_1>0$ and $\epsilon>0$, using the assumption $(C1)$, Lemma \ref{l3}, and since the functions $\varrho\rightarrow (t_1-\varrho)^{\gamma-1}I^{\beta}_{\varrho}E\|\upsilon(\varrho)\|$  and $\varrho\rightarrow (t_2-\varrho)^{\gamma-1}I^{\beta}_{\varrho}E\|\upsilon(\varrho)\|$ are Lebesgue integrable, as well as the continuity of the operator $\varPsi(t-\varrho,\varrho)$ in the uniform topology with respect to the variables $t$ and $\varrho$, where $0 \leq t\leq\ell$ and $0\leq \varrho\leq t-\epsilon$, we have 
	\begin{eqnarray*}
		I_{11}&\leq&\sup\limits_{\varrho\in [0,t_1-\epsilon]}\left\|\varPsi(t_2-\varrho,\varrho)-\varPsi(t_1-\varrho,\varrho)\right\|^2 C^2 \left(\int_{0}^{t_1-\epsilon}\int_{0}^{\varrho}(\varrho-s)^{\beta-1}E\left\|\upsilon(s)\right\|_Qdsd\varrho \right)^2\\
		&&\quad+C^4\left(\int_{t_1-\epsilon}^{t_1}\int_{0}^{\varrho}\left[(t_2-\varrho)^{\gamma-1}+(t_1-\varrho)^{\gamma-1}\right](\varrho-s)^{\beta-1}E\left\|\upsilon(\varrho)\right\|_Qdsd\varrho\right)^2\\
		&\leq&\sup\limits_{\varrho\in [0,t_1-\epsilon]}\left\|\varPsi(t_2-\varrho,\varrho)-\varPsi(t_1-\varrho,\varrho)\right\|^2 C^2m_1^2 \left(\int_{0}^{t_1-\epsilon}\int_{0}^{\varrho}(\varrho-s)^{\beta-1}E\left\|\upsilon(s)\right\|_Qdsd\varrho \right)^2\\
		&&\quad+C^4(\Gamma(\beta))^2\left(\int_{t_1-\epsilon}^{t_1}\left[(t_2-\varrho)^{\gamma-1}I^{\beta}_{\varrho}E\|\upsilon(\varrho)\|+(t_1-\varrho)^{\gamma-1}I^{\beta}_{\varrho}E\|\upsilon(\varrho)\|\right]d\varrho\right)^2\\
		&\rightarrow& 0~ \mbox{as }~ t_1\rightarrow t_2 \mbox{ and } \epsilon \rightarrow 0.
	\end{eqnarray*}
	For $I_{12}$, by using the assumption (C4), Lemma~\ref{l3} and since the function $\varrho\rightarrow (t_2-\varrho)^{\gamma-1}I^{\beta}_{\varrho}\mathcal{G}_r(\varrho)$ is Lebesgue integrable, we find that
	\begin{eqnarray*}
		I_{12}&\leq&C^4\left(\int_{t_1}^{t_2}\int_{0}^{\varrho}(t_2-\varrho)^{\gamma-1}(\varrho-s)^{\beta-1}\left(E\left\|g\left(s,\vartheta_{s},\vartheta(s)\right)\right\|^2\right)^{\frac{1}{2}}dsd\varrho \right)^2\nonumber\\
		&\leq&C^4\left(\int_{t_1}^{t_2}\int_{0}^{\varrho}(t_2-\varrho)^{\gamma-1}(\varrho-s)^{\beta-1}\left(\mathcal{G}_r(s)\right)^{\frac{1}{2}}dsd\varrho \right)^2\nonumber\\
		&\leq&C^4(\Gamma(\beta))^2\left(\int_{t_1}^{t_2}(t_2-\varrho)^{\gamma-1}I^{\beta}_{\varrho}\left(\mathcal{G}_r(\varrho)\right)^{\frac{1}{2}}d\varrho \right)^2\nonumber\\
		&\rightarrow& 0~ \mbox{as }~ t_2-t_1\rightarrow0.
	\end{eqnarray*}
	For $t_1=0$ and $0 <t_2\leq\ell$, we can easily prove that $I_{13}=0$. For $t_1>0$ and $\epsilon>0$, by using the assumption (C4), Lemma \ref{l3}, and since the functions $\varrho\rightarrow (t_1-\varrho)^{\gamma-1}I^{\beta}_{\varrho}\mathcal{G}_r(\varrho)$ and $\varrho\rightarrow (t_2-\varrho)^{\gamma-1}I^{\beta}_{\varrho}\mathcal{G}_r(\varrho)$ are Lebesgue integrable, as well as the continuity of the operator $\varPsi(t-\varrho,\varrho)$ in uniform topology about the variables $t$ and $\varrho$, where $0 \leq t\leq\ell$ and $0\leq \varrho\leq t-\epsilon$, we get 
	\begin{eqnarray*}
		I_{13}&\leq&\sup\limits_{\varrho\in [0,t_1-\epsilon]}\left\|\varPsi(t_2-\varrho,\varrho)-\varPsi(t_1-\varrho,\varrho)\right\|^2 C^2 \left(\int_{0}^{t_1-\epsilon}\int_{0}^{\varrho}(\varrho-s)^{\beta-1}\left(\mathcal{G}_r(s)\right)^{\frac{1}{2}}dsd\varrho \right)^2\\
		&&\quad+C^4\left(\int_{t_1-\epsilon}^{t_1}\int_{0}^{\varrho}\left[(t_2-\varrho)^{\gamma-1}+(t_1-\varrho)^{\gamma-1}\right](\varrho-s)^{\beta-1}\left(\mathcal{G}_r(s)\right)^{\frac{1}{2}}dsd\varrho\right)^2\\
		&\leq&\sup\limits_{\varrho\in [0,t_1-\epsilon]}\left\|\varPsi(t_2-\varrho,\varrho)-\varPsi(t_1-\varrho,\varrho)\right\|^2 C^2\left(\int_{0}^{t_1-\epsilon}\int_{0}^{\varrho}(\varrho-s)^{\beta-1}\left(\mathcal{G}_r(s)\right)^{\frac{1}{2}}dsd\varrho \right)^2\\
		&&\quad+C^4(\Gamma(\beta))^2\left(\int_{t_1-\epsilon}^{t_1}\left[(t_2-\varrho)^{\gamma-1}I^{\beta}_{\varrho}\left(\mathcal{G}_r(\varrho)\right)^{\frac{1}{2}}+(t_1-\varrho)^{\gamma-1}I^{\beta}_{\varrho}\left(\mathcal{G}_r(\varrho)\right)^{\frac{1}{2}}\right]d\varrho\right)^2\\
		&\rightarrow& 0~ \mbox{as }~ t_1\rightarrow t_2 \mbox{ and } \epsilon \rightarrow 0.
	\end{eqnarray*}
	For $I_{14}$, using the assumption $(C5)$, Lemma \ref{l3} and since the function $\varrho\rightarrow (t_2-\varrho)^{\gamma-1}I^{\beta}_{\varrho}E\|h(\varrho,\vartheta(\varrho))\|$ is Lebesgue integrable, we get
	\begin{eqnarray*}
		I_{14}&\leq&C^4Tr(Q)\left(\int_{t_1}^{t_2}\int_{0}^{\varrho}(t_2-\varrho)^{\gamma-1}(\varrho-s)^{\beta-1}E\|h(s,\vartheta(s))\|dsd\varrho \right)^2\nonumber\\
		&\leq&C^4(\Gamma(\beta))^2\left(\int_{t_1}^{t_2}(t_2-\varrho)^{\gamma-1}I^{\beta}_{\varrho}E\|h(\varrho,\vartheta(\varrho))\|d\varrho \right)^2\nonumber\\
		&\rightarrow& 0~ \mbox{as }~ t_2-t_1\rightarrow0.
	\end{eqnarray*}
	For $t_1=0$ and $0 <t_2\leq\ell$, we can easily prove that $I_{15}=0$. For $t_1>0$ and $\epsilon>0$, by  assumption (C5), Lemma \ref{l3}, and since the functions $\varrho\rightarrow (t_1-\varrho)^{\gamma-1}I^{\beta}_{\varrho}E\|h(\varrho,\vartheta(\varrho))\|$ and $\varrho\rightarrow (t_2-\varrho)^{\gamma-1}I^{\beta}_{\varrho}E\|h(\varrho,\vartheta(\varrho))\|$ are Lebesgue integrable, as well as the continuity of operator $\varPsi(t-\varrho,\varrho)$ in uniform topology with respect to the variables $t$ and $\varrho$, where $0 \leq t\leq\ell$ and $0\leq \varrho\leq t-\epsilon$, we have 
	\begin{eqnarray*}
		I_{15}&\leq&\sup\limits_{\varrho\in [0,t_1-\epsilon]}\left\|\varPsi(t_2-\varrho,\varrho)-\varPsi(t_1-\varrho,\varrho)\right\|^2 C^2 \left(\int_{0}^{t_1-\epsilon}\int_{0}^{\varrho}(\varrho-s)^{\beta-1}E\|h(s,\vartheta(s))\|dsd\varrho \right)^2\\
		&&\quad+C^4\left(\int_{t_1-\epsilon}^{t_1}\int_{0}^{\varrho}\left[(t_2-\varrho)^{\gamma-1}+(t_1-\varrho)^{\gamma-1}\right](\varrho-s)^{\beta-1}E\|h(s,\vartheta(s))\|dsd\varrho\right)^2\\
		&\leq&\sup\limits_{\varrho\in [0,t_1-\epsilon]}\left\|\varPsi(t_2-\varrho,\varrho)-\varPsi(t_1-\varrho,\varrho)\right\|^2 C^2\left(\int_{0}^{t_1-\epsilon}\int_{0}^{\varrho}(\varrho-s)^{\beta-1}E\|h(s,\vartheta(s))\|dsd\varrho \right)^2\\
		&&\quad+C^4(\Gamma(\beta))^2\left(\int_{t_1-\epsilon}^{t_1}\left[(t_2-\varrho)^{\gamma-1}I^{\beta}_{\varrho}E\|h(\varrho,\vartheta(\varrho))\|+(t_1-\varrho)^{\gamma-1}I^{\beta}_{\varrho}E\|h(\varrho,\vartheta(\varrho))\|\right]d\varrho\right)^2\\
		&\rightarrow& 0~ \mbox{as }~ t_1\rightarrow t_2 \mbox{ and } \epsilon \rightarrow 0.
	\end{eqnarray*}
	Hence $E\|(\mathcal{F}\vartheta)(t_2)-(\mathcal{F}\vartheta)(t_1) \|^2 \rightarrow 0$ as $t_1 \rightarrow t_2.$ Therefore, $\mathcal{F}(\Theta_r)$ is equicontinuous. The proof of the equicontinuity for the cases $t_1<t_2<0$ and $t_1<0<t_2$  is omitted as they are both quite straightforward. As a result, $\mathcal{F}(\Theta_r)$ is equicontinuous.                    
	
	\noindent{\bf Step 5:}  Let $\mathcal{Q} = \overline{conv}\mathcal{F}(\Theta_r)$, where $\overline{conv}$ means the closure of convex hull. We can easily verify that the function $\mathcal{F}$ maps $\mathcal{Q}$ into itself and the set $\mathcal{Q}\subset C(J,X)$ is equicontinuous. Now, we can prove that function $\mathcal{F}: \mathcal{Q}\rightarrow \mathcal{Q}$ is a $k$-set-contractive operator. From Lemma \ref{l8}, we know that for each subset $\mathcal{P}$ of $\mathcal{Q}$, there exists a countable subset $\mathcal{P}_0=\{\vartheta^n\}$ of $\mathcal{P}$ such that
	\begin{eqnarray}\label{3.10}
		\mu_C(\mathcal{F}(\mathcal{P})) \leq 2\mu_C(\mathcal{F}(\mathcal{P}_0).
	\end{eqnarray}
	The set $\mathcal{P}_0 \subset \mathcal{Q}$ is also equicontinuous as the set $\mathcal{Q}$ is  equicontinuous. Therefore by (\ref{3.7}), Lemmas \ref{l3},~\ref{l5},~\ref{l6},~\ref{l9}, and the assumptions (C3)-(C5), we have	
	\begin{eqnarray}\label{3.11}
		\mu\left((\mathcal{F}\{\vartheta^n(t)\})\right)&=&	\mu\left({A}^{-1}(0)\phi(0)-f(0,{A}^{-1}(0)\phi(0))+f(t,\{\vartheta^n_t(t)\})\right)\nonumber\\
		&&\quad+\mu\left(\displaystyle\int_{0}^{t}\varPsi(t-\varrho,\varrho)\mathcal{U}(\varrho)\phi(0)d\varrho\right)+\mu\left(\displaystyle\int_{0}^{t}\varPsi(t-\varrho,\varrho){B}\upsilon(\varrho)d\varrho \right)\nonumber\\
		&&\quad+ \mu\left(\displaystyle\int_{0}^{t} \varPsi(t-\varrho,\varrho)g\left(\varrho,\{\vartheta^n_{\varrho}(t)\},\{\vartheta^n(\varrho)\}\right)d\varrho\right) \nonumber  \\
		&&\quad+\mu\left(\displaystyle\int_{0}^{t}\varPsi(t-\varrho,\varrho)h(\varrho,\{\vartheta^n(\varrho)\})dw(\varrho)\right)\nonumber\\
		&&\quad+\mu\left(\displaystyle\int_{0}^{t}  \int_{0}^{\varrho}\varPsi(t-\varrho,\varrho)\varPhi(\varrho,s){B}\upsilon(s)dsd\varrho\right)\nonumber\\
		&&\quad+\mu\left(\displaystyle\int_{0}^{t}  \int_{0}^{\varrho}\varPsi(t-\varrho,\varrho)\varPhi(\varrho,s)g\left(s,\{\vartheta^n_s(t)\},\{\vartheta^n(s)\}\right)dsd\varrho\right) \nonumber \\ 
		&&\quad+\mu\left(\displaystyle\int_{0}^{t}  \int_{0}^{\varrho}\varPsi(t-\varrho,\varrho)\varPhi(\varrho,s)h(s,\{\vartheta^n(s)\})dw(s)d\varrho\right)\nonumber\\
		&\leq&\hat{L}_f\mu_C(\mathcal{P})+2m_1C\int_{0}^{t}(t-\varrho)^{\gamma-1}\mu_C(\upsilon(\varrho))d\varrho+2C\hat{L}_g\int_{0}^{t}(t-\varrho)^{\gamma-1}\mu_C(\mathcal{P})d\varrho\nonumber\\
		&&\quad+C\hat{L}_h\sqrt{\ell Tr(Q)}\int_{0}^{t}(t-\varrho)^{\gamma-1}\mu_C(\mathcal{P})d\varrho\nonumber\\
		&&\quad+4m_1C^2\displaystyle\int_{0}^{t}  \int_{0}^{\varrho}(t-\varrho)^{\gamma-1}(\varrho-s)^{\beta-1}\mu_C(\upsilon(s))dsd\varrho\nonumber\\
		&&\quad+4C^2\hat{L}_g\displaystyle\int_{0}^{t}  \int_{0}^{\varrho}(t-\varrho)^{\gamma-1}(\varrho-s)^{\beta-1}\mu_C(\mathcal{P})dsd\varrho\nonumber\\
		&&\quad+2C^2\hat{L}_h\sqrt{\ell Tr(Q)}\displaystyle\int_{0}^{t}  \int_{0}^{\varrho}(t-\varrho)^{\gamma-1}(\varrho-s)^{\beta-1}\mu_C(\mathcal{P})dsd\varrho. 
	\end{eqnarray}
	From (\ref{3.1}), we have 
	\begin{eqnarray}\label{3.12}
		\mu(\upsilon(t))&\leq&m_2\mu\Bigg[E\vartheta_1  + \int_{0}^{\ell}\kappa(s)dw(s) -{A}^{-1}(0)\phi(0)+f(0,{A}^{-1}(0)\phi(0))-f(t,\{\vartheta_t^n\})\nonumber \\
		&&\hspace{1cm}-\int_{0}^{t}\varPsi(t-\varrho,\varrho) \mathcal{U}(\varrho)\phi(0)d\varrho - \int_{0}^{t} \varPsi(t-\varrho,\varrho)g\left(\varrho,\{\vartheta^n_{\varrho}(t)\},\{\vartheta^n(\varrho)\}\right)d\varrho \nonumber \\
		&&\hspace{1cm}-\int_{0}^{t} \varPsi(t-\varrho,\varrho)h\left(\varrho,\{\vartheta^n(\varrho)\}\right)dw(\varrho)\nonumber\\
		&&\hspace{1cm}-\int_{0}^{t}  \int_{0}^{\varrho}\varPsi(t-\varrho,\varrho)\varPhi(\varrho,s)g\left(s,\{\vartheta^n_{s}(t)\},\{\vartheta^n(s)\}\right)dsd\varrho\nonumber \\
		&&\hspace{3cm}-\int_{0}^{t}  \int_{0}^{\varrho}\varPsi(t-\varrho,\varrho)\varPhi(\varrho,s)h\left(s,\{\vartheta^n(s)\}\right)dw(s)d\varrho \Bigg]\nonumber\\
		&\leq&m_2\Bigg[2C\hat{L}_g\int_{0}^{t}(t-\varrho)^{\gamma-1}\mu_C(\mathcal{P})d\varrho+C\hat{L}_h\sqrt{\ell Tr(Q)}\int_{0}^{t}(t-\varrho)^{\gamma-1}\mu_C(\mathcal{P})d\varrho\nonumber\\
		&&\quad+4C^2\hat{L}_g\displaystyle\int_{0}^{t}  \int_{0}^{\varrho}(t-\varrho)^{\gamma-1}(\varrho-s)^{\beta-1}\mu_C(\mathcal{P})dsd\varrho\nonumber\\
		&&\quad+2C^2\hat{L}_h\sqrt{\ell Tr(Q)}\displaystyle\int_{0}^{t}  \int_{0}^{\varrho}(t-\varrho)^{\gamma-1}(\varrho-s)^{\beta-1}\mu_C(\mathcal{P})dsd\varrho\Bigg]\nonumber\\
		&\leq&m_2\bigg[\left(2\hat{L}_g+\hat{L}_h\sqrt{\ell Tr(Q)}\right)\frac{C\ell^{\gamma}}{\gamma}+2\left(2\hat{L}_g+\hat{L}_h\sqrt{\ell Tr(Q)}\right)\frac{C^2\mathbf{B}(\gamma,\beta)\ell^{\beta+\gamma}}{\beta+\gamma}\bigg]\mu_C(\mathcal{P})\nonumber\\
		&=&m_2C{\ell}^{\gamma}\left(\frac{1}{\gamma}+\frac{2C\mathbf{B}(\gamma,\beta)}{\beta+\gamma} {\ell}^{\beta}\right)\left(2\hat{L}_g+\hat{L}_h\sqrt{\ell Tr(Q)}\right)\mu_C(\mathcal{P})\nonumber\\
		&=&m_2C{\ell}^{\gamma}\lambda_5\left(2\hat{L}_g+\hat{L}_h\sqrt{\ell Tr(Q)}\right)\mu_C(\mathcal{P}).  
	\end{eqnarray}
	From (\ref{3.11}), (\ref{3.12}) and Lemma~\ref{l5}, we have
	\begin{eqnarray}\label{3.13}
		\mu\left((\mathcal{F}\{\vartheta^n(t)\})\right)
		&\leq&\hat{L}_f\mu_C(\mathcal{P})+2m_1\frac{C\ell^{\gamma}}{\gamma}\times m_2C{\ell}^{\gamma}\lambda_5\left(2\hat{L}_g+\hat{L}_h\sqrt{\ell Tr(Q)}\right)\mu_C(\mathcal{P})\nonumber\\
		&&\quad+2\hat{L}_g\frac{C\ell^{\gamma}}{\gamma}\mu_C(\mathcal{P})+\hat{L}_h\sqrt{\ell Tr(Q)}\frac{C\ell^{\gamma}}{\gamma}\mu_C(\mathcal{P})\nonumber\\
		&&\quad+4m_1C^2\mathbf{B}(\gamma,\beta)\displaystyle\int_{0}^{t}(t-\varrho)^{\beta+\gamma-1}\mu_C(\upsilon(\varrho))d\varrho\nonumber\\
		&&\quad+4C^2\mathbf{B}(\gamma,\beta)\hat{L}_g\displaystyle\int_{0}^{t} (t-\varrho)^{\beta+\gamma-1}\mu_C(\mathcal{P})d\varrho\nonumber\\
		&&\quad+2C^2\mathbf{B}(\gamma,\beta)\hat{L}_h\sqrt{\ell Tr(Q)}\displaystyle\int_{0}^{t}  (t-\varrho)^{\beta+\gamma-1}\mu_C(\mathcal{P})d\varrho\nonumber\\	&\leq&\hat{L}_f\mu_C(\mathcal{P})+2C^2m_1m_2\frac{\ell^{2\gamma}}{\gamma}\lambda_5\left(2\hat{L}_g+\hat{L}_h\sqrt{\ell Tr(Q)}\right)\mu_C(\mathcal{P})\nonumber\\
		&&\quad+C\left(2\hat{L}_g+\hat{L}_h\sqrt{\ell Tr(Q)}\right)\frac{\ell^{\gamma}}{\gamma}\mu_C(\mathcal{P})\nonumber\\
		&&\quad+4C^2m_1\mathbf{B}\times m_2C{\ell}^{\gamma}\lambda_5\left(2\hat{L}_g+\hat{L}_h\sqrt{\ell Tr(Q)}\right)\frac{\ell^{\beta+\gamma}}{\beta+\gamma}\mu_C(\mathcal{P})	\nonumber\\
		&&\quad+4C^2\mathbf{B}(\gamma,\beta)\hat{L}_g\frac{\ell^{\beta+\gamma}}{\beta+\gamma}\mu_C(\mathcal{P})\nonumber\\
		&&\quad+2C^2\mathbf{B}(\gamma,\beta)\left(2\hat{L}_g+\hat{L}_h\sqrt{\ell Tr(Q)}\right)\frac{\ell^{\beta+\gamma}}{\beta+\gamma}\mu_C(\mathcal{P})\nonumber\\
		&\leq&\bigg[\hat{L}_f+C\ell^{\gamma}\lambda_5(2Cm_1m_2\lambda_5\ell^{\gamma}+1)\left(2\hat{L}_g+\hat{L}_h\sqrt{\ell Tr(Q)}\right)\bigg]\mu_C(\mathcal{P}).
	\end{eqnarray}
	From Lemma \ref{l7} and since the set $\mathcal{F}(\mathcal{P}_0)\subset \mathcal{P} $ is bounded and equicontinuous, we find that 
	\begin{eqnarray}\label{3.14}
		\mu_C(\mathcal{F}(\mathcal{P}_0))=\max\limits_{t\in J}\mu(\mathcal{F}(\mathcal{P}_0))(t).
	\end{eqnarray}
	Therefore, from (\ref{3.10}), (\ref{3.13}) and (\ref{3.14}), we have
	\begin{eqnarray}\label{3.15}
		\mu_C(\mathcal{F}(\mathcal{P}))\leq\bigg[\hat{L}_f+C\ell^{\gamma}\lambda_5(2Cm_1m_2\lambda_5\ell^{\gamma}+1)\left(2\hat{L}_g+\hat{L}_h\sqrt{\ell Tr(Q)}\right)\bigg] \mu_C(\mathcal{P}).
	\end{eqnarray}
	Hence, from (\ref{3.6}), (\ref{3.15}) and Definition \ref{d5}, we conclude that $\mathcal{F}: \mathcal{Q}\rightarrow \mathcal{Q}$ is a $k$-set contractive operator. The operator $\mathcal{F}$ defined by (\ref{3.7}) has at least one fixed point $\vartheta\in \mathcal{Q}$ from Lemma \ref{l11}, which is just a mild solution of time fractional non-autonomous system (\ref{1.1}) on interval $J$. Hence the system (\ref{1.1}) is controllable.
\end{proof}

\section{\textbf{Application}}
\noindent Consider the following non-autonomous fractional partial differential equation	\begin{eqnarray} \label{5.1}
	\left \{ \begin{array}{lll}
		^C\mathrm{D}^{\gamma}_t\left[\vartheta(x,t)-\hat{f}(t,\vartheta(x,t-h))\right]&\\\hspace{1cm}=\kappa(x,t)\Delta \left[\vartheta(x,t)-\hat{f}(t,\vartheta(x,t-h))\right]\vartheta(x,t)+\alpha\upsilon(x,t)
		&\\\hspace{1cm}+\hat{g}(t,\vartheta(x,t-h),\vartheta(x,t))+\hat{h}(t,\vartheta(x,t))\dfrac{dw}{dt}, ~h>0,~ x\in[0,\pi],~t\in [0,1],
		&\\\vartheta(0,t)=0,~\vartheta(\pi,t)=0~ \quad x \in [0, \pi], \\
		\vartheta(x,0)=(\kappa(\cdot,0))^{-1}\phi(x),
	\end{array}\right.
\end{eqnarray}
where $^C\mathrm{D}^{\gamma}_t$ denotes the partial Caputo derivative of order $0<\gamma<1,~J=[0,1],\:\: \Delta$ is the Laplace operator, $\kappa(x,t)$ is continuous function on $[0,\pi]\times[0,1]$  which is the coefficient of heat conductivity and it is also H\"{o}lder continuous in $t,$ that is for any $t_1,~t_2\in J$, there exist constants $0<\alpha\leq1$ and $C>0$ independent of $t_1$  and $t_2,$ such that
$$|\kappa(x,t_1)-\kappa(x,t_2)|\leq C|t_1-t_2|^{\alpha},\quad x \in [0,\pi],$$
and $\phi \in L^2([0,\pi],\mathbb{R}).$\\
Let $X=Y=U=L^2([0,\pi],\mathbb{R})$ be a Banach space defined with the $L^2$ norm $\|\cdot\|_2.$ Define a linear operator $A(t)$ in a Banach space $X$ by $A(t)\vartheta=-\kappa(x,t)\Delta \vartheta$
with the domain 
$D(A)= H^2(0,\pi)\cap H^1_0(0,\pi),$
where $ H^2(0,\pi)$ is the completion of the space $ C^2(0,\pi)$ with the norm 
$$\|u\|_{H^2(0,\pi)}=\bigg(\displaystyle\int_{[0,\pi]}\sum_{|\nu|\leq2}|D^{\nu}u(z)|^2dz\bigg)^{1/2}.$$ 
$C^2(0,\pi)$ is the set of all continuous functions defined on $[0,\pi]$ which have continuous partial derivatives of order less than or equal to $2.$ $H^1_0(0,\pi)$ represents the completion of $C^1(0,\pi)$ with  the norm  $\|u\|_{H^1(0,\pi)}$ and $C^1_0(0,\pi)$
is the set of all functions where $u\in C^1(0,\pi)$ with compact support on the domain $[0,\pi].$ From \cite{p1}, we know that $-A(t)$ generates an analytic semigroup $e^{-sA(t)}$ in $X$ satisfying $(B1)$ and $(B2)$.

\noindent For any $t\in[0,1],$ and $\varphi\in C((-\infty,\ell],X)$, we set $\vartheta(t)(x)=\vartheta(x,t)$ and $\upsilon(t)(x)=\upsilon(x,t),$ then \\$
\varphi(t)(x)=\vartheta_0(x,t),~~ t\in (-\infty,0],~f(t,\varphi)(x)=\hat{f}(t,\varphi(x,t)),~
g(t,\varphi,\vartheta(t))(x)=\hat{g}(t,\varphi(x,t),\vartheta(x,t)),$\\$h(t,\vartheta(t))(x)=\hat{f}(t,\vartheta(x,t)),~A^{-1}(0)=(\kappa(\cdot,0))^{-1}.$\\
Define the bounded linear
operator $B : U \rightarrow X$ by $B\upsilon(t)(x)=\alpha\upsilon (x,t)$.
We assume that the functions $f,~g,$ and $h$  satisfy the required assumptions. 
Then the problem (\ref{5.1}) can be written in an abstract form which satisfies all the assumptions Theorem \ref{t3.2}. Therefore applying Theorem \ref{t3.2}, we can conclude that system (\ref{5.1}) is controllable.

\section{\textbf{Conclusion}}
This paper is concerned with the controllability of a non-autonomous stochastic neutral differential equation of order $0<\gamma\leq1$ with infinite delay in an abstract space. We defined the mild solution using the probability density function, and then we examined the existence of mild solution and controllability of fractional stochastic system using the measure of non-compactness and the $k$-set contractive mapping. This work can be extended for the fractional-order non-autonomous system with variable orders.

\noindent \textbf{Acknowledgment}
The first and third authors acknowledge UGC, India, for providing financial support through MANF F.82-27/2019 (SA-III)/191620066959 and F.82-27/2019 (SA-III)/ 4453, respectively.

\end{document}